\documentclass[10pt,reqno]{amsart}
\usepackage{amssymb,latexsym}
\usepackage{amsmath,color}
\usepackage{amsthm}
\usepackage{epsfig}
\usepackage{graphicx}
\usepackage{titletoc}
\usepackage{dsfont}
\usepackage{nccbbb}
\usepackage{amsmath}
\usepackage{amssymb}
\usepackage{amssymb}
\usepackage{amstext}
\usepackage{amsopn}
\usepackage{amsbsy}
\usepackage{amscd}
\usepackage{amsxtra}
\usepackage{accents}
\usepackage{bbm}
\usepackage{yfonts}
\usepackage{bbm}
\usepackage[pagewise]{lineno}
\usepackage[dvipsnames]{xcolor} 
\definecolor{darkgreen}{rgb}{0.0, 0.5, 0.0}

\usepackage[colorlinks=true, linkcolor=blue, citecolor=magenta, urlcolor=cyan]{hyperref}

\numberwithin{equation}{section}

\usepackage{xparse}
\usepackage{microtype}

        \makeatletter
        \newcommand{\ovset}[3][0ex]{%
          \mathrel{\mathop{#3}\limits^{
            \vbox to#1{\kern0\ex@
            \hbox{$\scriptstyle#2$}\vss}}}}
        \makeatother

\NewDocumentCommand{\hideparbox}{O{c}mm+m}
 {
  \group_begin:
  \vbox_set:Nn \l_hideparbox_box
   {
    \use:c { @parboxrestore }
    \hsize=#3\scan_stop:
    \strut#4\par
   }
  \vbadness=\c_ten_thousand 
  \vbox_set_split_to_ht:NNn \l_hideparbox_box \l_hideparbox_box { #2 }
  \parbox[#1][#2]{#3}
   {
    \vbox_unpack:N \l_hideparbox_box
   }
  \group_end:
 }
\ExplSyntaxOff

\usepackage{cases}

\usepackage[T1]{fontenc}

\newtheorem{theorem}{Theorem}[section]
\newtheorem{proposition}[theorem]{Proposition}
\newtheorem{lemma}[theorem]{Lemma}
\newtheorem{corollary}[theorem]{Corollary}

\theoremstyle{definition}

\newtheorem{remark}{Remark}

\usepackage{esint}

\newcommand{\ds}{\displaystyle}

\newcommand{\eps}{\varepsilon}

\newcommand{\dxx}{\,\mathrm{d}{x}}

\newcommand{\e}{\varepsilon}

\usepackage[misc]{ifsym}

\begin{document}
\title[Boundary layers in multiplicatively nonlocal elliptic equations]{}
\small


\begin{center}
{\LARGE\bf Refined boundary layer asymptotics for elliptic equations\\ with multiplicative nonlocal effects
}\vspace{3mm}\\
\large Chiun-Chang Lee\footnote{Institute for Computational and Modeling Science,  National Tsing Hua University,  Hsinchu 30013, Taiwan.\\
\hspace*{6mm}\Letter\hspace*{1.5mm}{ lee2@mx.nthu.edu.tw }
} 
 $\cdot$ Sang-Hyuck Moon\footnote{Department of Mathematics and Institute of Mathematical Science, Pusan National
University, Busan 46241, Republic of Korea\\
{\hspace*{6mm}\Letter\hspace*{1.5mm}{ shmoon@pusan.ac.kr}}} 
$\cdot$ Wen Yang\footnote{Department of Mathematics, Faculty of Science and Technology, University of
Macau, Taipa, Macau.\\
{\hspace*{6mm}\Letter\hspace*{1.5mm}{wenyang@um.edu.mo}}
}

\end{center}
\vspace{-1.5cm}
\maketitle

\begin{abstract} 
We investigate singularly perturbed elliptic problems with multiplicative nonlocal diffusion terms subject to Robin boundary conditions. The diffusion depends on a global quantity of the solution, which introduces a nonlocal coupling between the global behavior of the solution and the boundary asymptotics. As the perturbation parameter tends to zero, we establish precise asymptotic expansions of the solutions that capture the structure of boundary layers coupled with the multiplicative nonlocal diffusion effect. Moreover, the interaction between the nonlocal diffusion and the boundary geometry manifests as refined higher-order terms wherein geometric quantities, such as the mean curvature, appear explicitly; our analysis thus quantifies the influence of global coupling on the boundary layer structure, extending classical singular perturbation theory to multiplicative nonlocal frameworks.
\vspace{2mm}
\end{abstract}
\section{\bf Introduction and statement of the main results}
\noindent

In various physical and biological modelling contexts, certain singularly perturbed models describing boundary layer phenomena incorporate \emph{nonlocal} effects through global coefficients that depend on the solution and enter in a \emph{multiplicative} manner~\cite{CF2011,L2019,LMWY2025,TBL2006,Z2017}. In contrast, in additive nonlocal models (e.g., Fredholm integro-differential equations~\cite{CLW2025}), the nonlocal terms typically arise as lower-order sources. By comparison, multiplicative nonlocal models involve global quantities directly in the diffusion coefficient, thereby altering the principal part of the differential operator through a solution-dependent coefficient (cf.~\cite{DS2010}).

As the singular perturbation parameter tends to zero, multiplicative nonlocal effects induce nontrivial boundary layers, whose structure is determined by the interplay between the domain geometry and the global nonlocal dependence of the solution. From a theoretical perspective, combining nonlocality with singular perturbation brings major analytical challenges, notably the absence of a natural variational structure~\cite{CSH2017,CC2009,KS2015,L2021,L2019N}. The present work aims to develop a rigorous asymptotic framework for a specific class of singularly perturbed  elliptic equations with multiplicative nonlocal structures. 

Specifically, let $\Omega \subset \mathbb{R}^n$ be a bounded, connected domain with a smooth boundary $\partial\Omega$. We investigate the asymptotic behavior of solutions to the singularly perturbed elliptic equation
\begin{equation}\label{maineq}
\varepsilon^{2}\mathcal{A}\left( \fint_{\Omega} q(u)\,\mathrm{d}x \right) \Delta u = f(u) \quad \text{in } \Omega,
\end{equation}
where $q, f: \mathbb{R} \to \mathbb{R}$ and $\mathcal{A}: \mathbb{R} \to \mathbb{R}^{+}$ are given functions. Throughout this work, we adopt the notation $\fint_{\Omega} := \frac{1}{|\Omega|} \int_{\Omega}$ to denote the average integral, where $|\Omega|$ represents the Lebesgue measure of $\Omega$ in $\mathbb{R}^n$.  The solution $u$ depends on the singular perturbed parameter~$0<\varepsilon\ll1$, it should properly be denoted by $u_{\varepsilon}$; however, this subscript will be omitted whenever no confusion arises. For \eqref{maineq}, we consider the Robin boundary condition
\begin{equation}\label{bd}
u+\varepsilon\gamma\,\partial_{\vec n}u=b_0 \qquad \text{on }\partial\Omega.
\end{equation}
Here $\gamma \ge 0$ and the boudnary data~$b_0$ are constants independent of $\varepsilon$. The operator $\partial_{\vec n}=\vec n\cdot\nabla$ denotes the outward normal derivative on $\partial\Omega$. Notably, the balanced scaling of $\varepsilon$ in both \eqref{maineq} and \eqref{bd} plays a crucial role in our analysis, as it leads to a rich variety of asymptotic behaviors~\cite{f1973}. 

\subsection{Modeling background and assumptions}  The multiplicative nonlocal structure of \eqref{maineq} arises in a wide range of applications. Prototypical examples include Carrier-type equations, where the diffusion coefficient is modulated by the $\text{L}^p$ norm of the solution, representing effects such as total mass or nonlinear beam deflections. Similar structures appear in stationary Keller--Segel systems, where global coupling terms like $\int_{\Omega} u\text{e}^u \,\mathrm{d}x$ or $\int_{\Omega} u^p \,\mathrm{d}x$ ($u \ge 0$ and $p>1$) regulate chemotactic aggregation~\cite{CLWY2026, LWY2020}. Such models are also prevalent in electrochemistry and semiconductor theory, notably in nonlocal Poisson--Boltzmann and sinh-Gordon equations. In these contexts, the singular perturbation parameter $\varepsilon$ is typically associated with the nanometer-scale Debye length, and the resulting boundary layers characterize the formation of electrical double layers~\cite{s2012, s2021, WL2014}. 

These diverse physical contexts lead to nonlocal terms of the form 
$\mathcal{A}\left(\fint_\Omega q(u)\,\mathrm{d}x\right)$, 
which introduce a multiplicative nonlocal perturbation into the singularly perturbed problem. 
Here the normalization factor $|\Omega|^{-1}$ naturally arises from dimensionless formulations. 
From a mathematical viewpoint, however, this factor can be absorbed into the function $q$, 
for instance by replacing $q(u)$ with $|\Omega|\,q(u)$, without affecting the essential features of the analysis. To ensure a rigorous derivation of the refined asymptotic expansions, we impose the following structural and regularity assumptions on the nonlinearities $q, \mathcal{A},$ and $f$:
\begin{equation}\label{q-A}
q\in \mathrm{C}^{1}(\mathbb{R};\mathbb{R}), \qquad \mathcal{A}\in \mathrm{C}^{3}(\mathbb{R};\mathbb{R}^{+}),
\end{equation}
and 
\begin{equation}\label{assume-f}
f \in \mathrm{C}^2(\mathbb{R}; \mathbb{R}), \quad \inf_{\mathbb{R}}f' > 0, \quad \text{and} \quad f(0) = 0.
\end{equation}
The conditions on $f$ are consistent with the aforementioned physical models, while the regularity requirements for $q$ and $\mathcal{A}$ provide the necessary analytical framework for capturing the delicate interaction between the boundary layer profiles and the global nonlocal effects as $0 < \varepsilon \ll 1$.

Our primary interest lies in boundary layer phenomena. Under the Robin boundary condition~\eqref{bd}, solutions $u$ develop boundary layers as $\eps\downarrow0$, while the quantity $\fint_{\Omega}q(u)\,\mathrm{d}x$ acts as an implicit parameter depending on both $\varepsilon$ and $\gamma$. Consequently, the asymptotic behavior of the nonlocal coefficient and that of the solution are strongly coupled. Since boundary layer structures are known to depend sensitively on the geometry of the domain, we aim to investigate how geometric features of $\Omega$, such as curvature effects, enter the refined asymptotics. In particular, we focus on the uniqueness of solutions for sufficiently small $\varepsilon$ and on establishing higher-order asymptotic expansions that explicitly reveal the influence of domain geometry.

\subsection{Uniqueness}

We first establish the uniqueness of the solution $u_\varepsilon$ to \eqref{maineq}--\eqref{bd} and characterize its leading-order asymptotic behavior as $\eps \downarrow 0$.

\begin{proposition}\label{existence}
Let $\Omega \subset \mathbb{R}^n$ be a connected, bounded domain with a smooth boundary $\partial\Omega$, and let $\gamma \geq 0$. Under the assumptions \eqref{q-A} and \eqref{assume-f}, there exists a  constant $\boldsymbol{\mathfrak{e}^*} > 0$ such that for each $\varepsilon \in (0, \boldsymbol{\mathfrak{e}^*})$, the problem \eqref{maineq} with boundary condition \eqref{bd} admits a unique solution $u \in \mathrm{C}^{2,\alpha}(\Omega) \cap \mathrm{C}^{1,\alpha}(\overline{\Omega})$, for some $\alpha\in(0,1)$. This solution satisfies the following exponential decay estimate:
    \begin{align}\label{int-est-u}
        |u(x)| \le b_0 \exp \left( - \frac{C_* \mathrm{dist}(x, \partial\Omega)}{\varepsilon} \right), \quad \text{for } x \in \Omega,
    \end{align}
where $C_*>0$ is a constant independent of $\eps$, and  $\mathrm{dist}(x, \partial\Omega)$ denotes the distance from $x$ to the boundary $\partial\Omega$. Moreover, 
the leading-order asymptotics of $u$ 
as $0<\eps\ll1$ is given by
    \begin{align}\label{u-to-W}
        \max_{x \in \overline{\Omega}} \left| u(x) - W \left( \frac{\mathrm{dist}(x, \partial\Omega)}{\varepsilon} \right) \right| \le C_{**} \varepsilon,
    \end{align}
    where $W$ is the unique solution to the ordinary differential equation
    \begin{align}\label{ode-W}
    \begin{cases}
    \mathcal{A}(q(0))W''(t) = f(W(t)), & t > 0, \\
    W(0) - \gamma W'(0) = b_0, & \lim\limits_{t \to \infty} W(t) = 0.
    \end{cases}
    \end{align}
\end{proposition}
In Section~\ref{prop1-sec}, we refer the reader to Lemma~\ref{lem-W} for the properties of $W$ and present the proof of Proposition~\ref{existence}. 

We outline the heuristic idea behind Proposition~\ref{existence}. To establish the uniqueness of $u$, we introduce a family of solutions to the corresponding \textit{local-type} equations. For each fixed ${\theta} > 0$, let $v_{{\theta},\varepsilon}$ satisfy
\begin{align}\label{vy-eq}
\begin{cases}
    {\theta}^2 \Delta v_{{\theta},\varepsilon} = f(v_{{\theta},\varepsilon}) & \text{in } \Omega, \\
    v_{{\theta},\varepsilon} + \gamma \varepsilon \partial_{\vec{n}} v_{{\theta},\varepsilon} = b_0 & \text{on } \partial\Omega.
\end{cases}
\end{align}
Since $f:\mathbb{R}\to\mathbb{R}$ is strictly increasing (cf.~\eqref{assume-f}) and $\gamma \ge 0$, problem~\eqref{vy-eq} admits a unique solution $v_{{\theta},\varepsilon}$ for each ${\theta}>0$ and $\varepsilon>0$. Moreover, the comparison principle yields $\min\{0,b_0\}\leq v_{{\theta},\varepsilon}\leq \max\{0,b_0\}$ in $\overline{\Omega}$. In view of this uniqueness, we introduce the consistency mapping
$\boldsymbol{\mathsf{M}}_{\varepsilon}\colon (0,\infty)\to\mathbb{R}$ defined by
\begin{align}\label{vy-map}
\boldsymbol{\mathsf{M}}_{\varepsilon}({\theta})
:=
{\theta}^2
-
\varepsilon^2
\mathcal{A}\left(
\fint_\Omega q(v_{{\theta},\varepsilon})\,\mathrm{d}x
\right).
\end{align}
Consequently, any solution $u$ of problem~\eqref{maineq}--\eqref{bd} corresponds to a root of the equation~$\boldsymbol{\mathsf{M}}_{\varepsilon}({\theta})=0.$

Furthermore, by standard elliptic theory (cf.~\cite[Theorem~1.1]{T2000}), the solution $v_{{\theta},\varepsilon}$ depends smoothly on the parameter ${\theta}>0$. Consequently, $\boldsymbol{\mathsf{M}}_{\varepsilon}$ is a $\mathrm{C}^1$ function on $(0,\infty)$. To establish the uniqueness of $u$ in the problem~\eqref{maineq}--\eqref{bd}, we analyze the derivative of $\boldsymbol{\mathsf{M}}_{\varepsilon}$ with respect to $\theta$. In particular, by determining the sign of this derivative for a sufficiently small $\varepsilon$, we show that there exists a positive constant $\boldsymbol{\mathfrak{e}^*}$ such that for every $\varepsilon \in (0,\boldsymbol{\mathfrak{e}^*})$, the function $\boldsymbol{\mathsf{M}}_{\varepsilon}$ has a unique zero, denoted by ${\theta}(\varepsilon)$. Consequently, $u=v_{{\theta}(\varepsilon),\varepsilon}$ is the unique solution of \eqref{maineq}--\eqref{bd} for all sufficiently small $\varepsilon$. The asymptotic estimates \eqref{int-est-u} and \eqref{u-to-W} then follow directly from the comparison principle.

 We remark that if either of the following conditions holds:
\begin{equation}\label{ass-i-ii}
\mathrm{(i)}\, \mathcal{A}' \le 0 \text{ and } q' \ge 0; \quad \text{or} \quad \mathrm{(ii)}\, \mathcal{A}' \ge 0 \text{ and } q' \le 0,
\end{equation}
Then we can apply the comparison principle to show that problem \eqref{maineq}--\eqref{bd} admits a unique solution $u$ for all $\varepsilon > 0$ (see Remark~\ref{rk-prop1}). However, if neither (i) nor (ii) is satisfied, the comparison principle alone is no longer sufficient to ensure uniqueness. Since the primary objective of this work is to characterize the asymptotic behavior of $u$ as $\varepsilon \to 0^+$, we shall not further discuss the uniqueness or multiplicity of solutions for general functions $\mathcal{A}$ and $q$ when $\varepsilon$ is not small.

\subsection{Preliminaries on nonlocal and geometric effects}\label{sec-sp}
Before presenting the main theorems, we describe the nonlocal and geometric effects governing the asymptotic behavior of the unique solution $u$ in problem \eqref{maineq}--\eqref{bd} as $\varepsilon \downarrow 0$. 
Since the analysis involves several auxiliary quantities, we introduce the necessary notation and preliminary results in this section. 
The main theorems will be stated in Section~\ref{s-mthm}.

 Since $b_0 = 0$ leads to the trivial solution $u \equiv 0$ for $\varepsilon \in (0, \boldsymbol{\mathfrak{e}^*})$, we assume, without loss of generality, that
\begin{align}\label{a-0}
    b_0 > 0.
\end{align}
It follows from \eqref{u-to-W} and \eqref{ode-W} that $u$ converges uniformly to $b_*$ on $\partial\Omega$ as $\varepsilon \downarrow 0$, where the boundary value $b_* := W(0) \in (0, b_0]$ is the unique solution of the algebraic equation
\begin{align}\label{b0}
b_* + \gamma \sqrt{\frac{2F(b_*)}{\mathcal{A}(q(0))}} = b_0,
\end{align}
with
\begin{equation}\label{big-F}
F(t):=\int_0^t f(s)\,\mathrm{d}s.
\end{equation}

This relation is derived by substituting the boundary condition $W(0) - \gamma W'(0) = b_0$ into the first integral of \eqref{ode-W}, which yields the explicit derivative $W'(0)=-\sqrt{\frac{2F(W(0))}{\mathcal{A}(q(0))}}$.
Consequently, $u$ develops a boundary layer near the boundary $\partial\Omega$ as $\varepsilon$ tends to zero.

Note that the coefficient $\mathcal{A} \!\left( \fint_\Omega q(u)\,\mathrm{d}x \right)$ depends on both $\varepsilon$ and the domain $\Omega$ through the nonlocal dependence on the solution $u$. The leading-order asymptotics \eqref{u-to-W} alone is insufficient to capture how the nonlocal term and the domain geometry influence the boundary layer structure of $u$. Indeed, although the estimate \eqref{int-est-u} implies
$\lim_{\eps\downarrow0}
\mathcal{A}\left(\fint_\Omega q(u)\,\mathrm{d}x\right)
=
\mathcal{A}(q(0)),$
this limit does not distinguish $u$ from the solution $\widetilde{u}$ of the corresponding local-type equation $\varepsilon^2 \mathcal{A}(q(0))\Delta\widetilde{u}=f(\widetilde{u}).$ Therefore, it is necessary to characterize the nonlocal perturbation by deriving a refined asymptotic expansion of $\mathcal{A}\left(\fint_\Omega q(u)\,\mathrm{d}x\right)$ beyond the constant $\mathcal{A}(q(0))$, which is essential for revealing the geometric effects in the boundary layer.

To capture the influence of the domain geometry on the asymptotic structure of $u$, we derive a higher-order expansion of $\fint_\Omega q(u)\,\mathrm{d}x$ as $\varepsilon \downarrow 0$. For this purpose, we introduce the geometric properties and the quantities appearing in the asymptotic expansions that will be used throughout the analysis.

\begin{itemize}
\item[\bf(G)] {\bf Geometric Properties Review.} Denote by $|\partial\Omega|$ the $(N-1)$-dimensional surface measure of $\partial\Omega$.  For $x\in\partial\Omega$, let $\kappa_i(x)$, $i=1,\ldots,N-1$, be the principal curvatures at $x$. The mean curvature of $\partial\Omega$ is given by
\begin{align*}
\mathcal{H}_{\partial\Omega}(x):=\frac{1}{N-1}\sum_{i=1}^{N-1}\kappa_i(x), \quad x\in\partial\Omega.
\end{align*}
It is known (cf. \cite[Theorem~1]{F1984} and \cite[Lemma~14.16]{GT2001}) that if $\partial\Omega\in\mathrm{C}^{4}$, then there exists $d_0>0$ depending on $\Omega$ such that for each $d\in(0,d_0)$, there hold $\Gamma_d\in\mathrm{C}^{4}$, where we define
\begin{align}\label{o-n3}
    \Omega_d:= \{x \in \Omega: \ \text{dist}(x,\partial\Omega)\in(0,d)\},\,\,\Gamma_d:= \{x \in \overline{\Omega}: \, \text{dist}(x,\partial\Omega)=d\}.
\end{align}
Moreover, we have that:
\begin{itemize}
    \item[\bf(i)] For each $x\in\Omega_{d_0}$, the nearest point $\sigma(x)\in\partial\Omega$ of $x$ to the boundary~$\partial\Omega$  is unique, and
    \begin{align}\label{x-to=sigma}
    x=\sigma(x)-\mathrm{dist}(x,\partial\Omega)\vec{n}(\sigma(x)).    
    \end{align}
    \item[\bf(ii)]  The mapping $x\mapsto(\delta(x),\sigma(x))$ is a $\mathrm{C}^{4}$-diffeomorphism from $\Omega_{d_0}$ onto $(0,d_0)\times\partial\Omega$, where $\delta(x)=\mathrm{dist}(x,\partial\Omega)$. Moreover, we choose
    \begin{align}\label{d-star}
        d_{*}=\min\left\{d_0,\frac{1}{1 + 2 \sup\limits_{\sigma \in \partial\Omega}
\max\limits_{1 \le i \le N-1} |\kappa_i(\sigma)|}\right\}\in(0,d_0].
    \end{align}
    Then, in terms of the curvilinear coordinates~$(\delta,\sigma)$, the Laplace operator in $\Omega_{d_{*}}$ is definitely represented as 
    \begin{align}\label{Laplace}
        \Delta=\frac{\partial^2}{\partial\delta^2}-(N-1)\mathcal{H}_{\Gamma_{\delta}}(\sigma)\frac{\partial}{\partial\delta}+\Delta_{\Gamma_{\delta}},\,\,(\delta,\sigma)\in(0,d_{*})\times\partial\Omega,
    \end{align}
    where
    \begin{align}\label{me-H}
        \mathcal{H}_{\Gamma_{\delta}}(\sigma):=\frac{1}{N-1}\sum_{i=1}^{N-1}\frac{\kappa_i(\sigma)}{1-\kappa_i(\sigma)\delta}
    \end{align}
    is the mean curvature of $\Gamma_{\delta}$ measured at $\sigma-\delta\vec{n}(\sigma)$, and $\Delta_{\Gamma_{\delta}}$ stands for the Beltrami--Laplacian on $\Gamma_{\delta}$.
\end{itemize}
\item[\bf(N)] 
We introduce several quantities and notations arising in the asymptotic expansions.

\begin{itemize}
\item[\bf (i)] 
Let $\Phi$ and $\Psi$ be the unique solutions to the following ODE problems:
\begin{equation}\label{Phi-new}
\left\{
\begin{aligned}
&\mathcal{A}(q(0))\Phi''(t)= f'(W(t))\Phi(t)+\mathcal{A}(q(0))f(W(t)), \quad t>0,\\
&\Phi(0)-\gamma \Phi'(0)=0, \quad \text{and} \quad \Phi(t)\to 0 \ \text{as } t\to\infty,
\end{aligned}
\right.
\end{equation}
and
\begin{equation}\label{Psi-new}
\left\{
\begin{aligned}
&\mathcal{A}(q(0))\Psi''(t)= f'(W(t))\Psi(t)-\sqrt{2\mathcal{A}(q(0))F(W(t))}, \quad t>0,\\
&\Psi(0)-\gamma \Psi'(0)=0, \quad \text{and} \quad \Psi(t)\to 0 \ \text{as } t\to\infty,
\end{aligned}
\right.
\end{equation}
respectively. The properties of $\Phi$ and $\Psi$ will be established in Lemma~\ref{lem-PhiPsi}.

\item[\bf (ii)] 
Define the function $\mathcal{Q}_F$ by
\begin{align}\label{Q-F}
\mathcal{Q}_F(t):=\int_0^t\frac{q(s)-q(0)}{\sqrt{2F(s)}}\,\mathrm{d}s, \quad t\in[0,\infty),
\end{align}
where $F$ is given in \eqref{big-F}. By \eqref{assume-f}, we have $\mathcal{Q}_F \in C^1([0,\infty))$ and
$\mathcal{Q}_F'(0)=\frac{q'(0)}{\sqrt{f'(0)}}.$ Furthermore, define
\begin{align}\label{I-J}
\begin{cases}
\displaystyle \boldsymbol{\mathcal{I}}_{W,\Phi}:= \int_0^{\infty} q'(W(s))\Phi(s)\,\mathrm{d}s,\\[6pt]
\displaystyle \boldsymbol{\mathcal{J}}_{W,\Psi}:= \int_0^\infty\left(\sqrt{\mathcal{A}(q(0))}\,\mathcal{Q}_F(W(s)) - q'(W(s))\Psi(s)\right)\mathrm{d}s,
\end{cases}
\end{align}
where $W$, $\Phi$, and $\Psi$ are the unique solutions to \eqref{ode-W}, \eqref{Phi-new}, and \eqref{Psi-new}, respectively. 
The quantities $\boldsymbol{\mathcal{I}}_{W,\Phi}$ and $\boldsymbol{\mathcal{J}}_{W,\Psi}$ are well defined in view of the estimates \eqref{w3w} and \eqref{cw-exp}.

\item[\bf (iii)] We write $a\lesssim b$ if there exists a positive constant $C$ such that $a\leq Cb$, where $C$ denotes a generic constant that may vary from line to line. We use the standard big-$O$ notation as $\varepsilon \downarrow 0$, and write $o_{\varepsilon}(1)$ for quantities that vanish in this limit, without specifying their precise rate.
\end{itemize}
\end{itemize}

%

\subsection{Statement of the main result}\label{s-mthm}

To derive refined asymptotics of $u$ as $\varepsilon \downarrow 0$, we establish a detailed asymptotic expansion of $\fint_\Omega q(u)\,\mathrm{d}x$. 
The following theorem identifies a key novel feature in this expansion, namely that both the boundary mean curvature $\mathcal{H}_{\partial\Omega}$ and the geometric quantity  $\frac{|\partial\Omega|}{|\Omega|}$ (surface-to-volume ratio) enter explicitly.

\begin{theorem}[Nonlocal perturbation involving geometric effects]\label{thm-nonlocal} Under the same hypotheses as in Proposition~\ref{existence}, we assume, without loss of generality, that \eqref{a-0} holds. Then, for arbitrarily small $\eps>0$ we have the following regularly perturbed expansion involving the domain geometry:
\begin{equation}\label{q-three}
  \begin{aligned}
    \fint_{\Omega} q(u) \dxx=&\,q(0)+\eps\frac{|\partial\Omega|}{|\Omega|}\sqrt{\mathcal{A}(q(0))}\mathcal{Q}_F(b_*)\\
    &\,-\eps^2\frac{|\partial\Omega|^2}{|\Omega|^2}\left(\frac{\mathcal{A}'(q(0))\mathcal{Q}_F(b_*)}{\left(\mathcal{A}(q(0))\right)^{\frac32}}{\boldsymbol{\mathcal{I}}}_{W,\Phi}+(N-1)\frac{|\Omega|\int_{\partial\Omega}\mathcal{H}_{\partial\Omega}\,\mathrm{d}\sigma}{|\partial\Omega|^2}{\boldsymbol{\mathcal{J}}}_{W,\Psi}+\boldsymbol{\mathrm{o}}_\eps\right),
\end{aligned}  
\end{equation}
where $b_*\in(0,b_0]$ is uniquely determined by \eqref{b0}, and  all quantities in \eqref{q-three} are defined by (G) and (N).
\end{theorem}
It should be emphasized that \eqref{q-three} represents the ``best-fit form,'' since both $\frac{|\partial\Omega|}{|\Omega|}$ and the boundary mean curvature $\mathcal{H}_{\partial\Omega}$ carry the physical dimension of inverse length, whereas the quantity $\frac{|\Omega|}{|\partial\Omega|^2}\int_{\partial\Omega}\mathcal{H}_{\partial\Omega}\,\mathrm{d}\sigma$
is dimensionless. More precisely, let $L$ denote the physical dimension of length, and let $[\,\cdot\,]$ indicate the dimension of a given physical quantity. Then $[|\Omega|] = L^N,~ [|\partial\Omega|] = L^{N-1},~ [\mathcal{H}_{\partial\Omega}] = L^{-1}.$ It follows that
$\left[\frac{|\Omega|\int_{\partial\Omega}\mathcal{H}_{\partial\Omega}\,\mathrm{d}\sigma}{|\partial\Omega|^2}\right]
= \frac{L^N \cdot L^{-1} \cdot L^{N-1}}{L^{2N-2}} = 1,$
confirming that this combination is indeed dimensionless.

Recall \eqref{x-to=sigma} and \eqref{me-H}. 
In view of Theorem~\ref{thm-nonlocal}, we are now prepared to present the main result concerning the refined asymptotics of $u$ near the boundary as $\varepsilon \downarrow 0$.

\begin{theorem}\label{mm-thm}
Let $\Omega_{d_*}$, defined by \eqref{o-n3}, be a subdomain of $\Omega$, where $d_*>0$ satisfies \eqref{d-star}. Under the same hypotheses as in Proposition~\ref{existence}, we further assume \eqref{a-0}. Then 
$\sup_{\Omega\setminus\overline{\Omega_{d_*}}}\left(|u|+\eps|\nabla{u}|\right)\xrightarrow{\eps\downarrow0}0$ exponentially.
 Moreover, we have
the following asymptotic expansions which are uniformly in $\overline{\Omega_{d_*}}$:
\begin{align}\label{thm-u(i)}
\frac{1}{\eps}\left(u(x)- W(\frac{\delta(x)}{\e})\right)+ \frac{|\partial\Omega|\mathcal{A}'(q(0))\mathcal{Q}_F(b_*)}{|\Omega|\left(\mathcal{A}(q(0))\right)^{\frac32}} \Phi(\frac{\delta(x)}{\eps})- (N-1) \mathcal{H}_{\Gamma_{\delta(x)}}(\sigma(x))\Psi(\frac{\delta(x)}{\eps} )\xrightarrow{\eps\downarrow0}0,
\end{align}
and
\begin{equation}\label{thm-u(ii)}
\begin{aligned}
\partial_{\vec{n}}{u(x)}-\frac{1}{\eps}\sqrt{\frac{2}{\mathcal{A}(q(0))}F(W(\frac{\delta(x)}{\eps}))} -&\, \frac{|\partial\Omega|\mathcal{A}'(q(0))\mathcal{Q}_F(b_*)}{|\Omega|(\mathcal{A}(q(0)))^{\frac32}} \Phi'(\frac{\delta(x)}{\eps}) \\
+&\, (N-1) \mathcal{H}_{\Gamma_{\delta(x)}}(\sigma(x)) \Psi'(\frac{\delta(x)}{\eps})\xrightarrow{\eps\downarrow0}0,   
\end{aligned}
\end{equation}
where $\vec{n}=\vec{n}(x)\parallel\vec{n}(\sigma(x))$ is the unit normal vector at $x\in\Gamma_{\delta(x)}$  pointing to the boundary~$\partial\Omega$.
\end{theorem}

Theorem~\ref{mm-thm} reveals in a precise manner how the geometry of the domain influences the boundary layer structure of $u$. In particular, \eqref{thm-u(i)} and \eqref{thm-u(ii)} provide refined asymptotic descriptions of $u(x_{\varepsilon})$ and $\partial_{\vec{n}}u(x_{\varepsilon})$ for points $x_{\varepsilon}\in\overline{\Omega_{d_*}}$ approaching the boundary, subject to the condition $\lim_{\varepsilon\downarrow 0}\frac{\delta(x_{\varepsilon})}{\varepsilon}<\infty.$ These results capture the delicate interplay between the singular perturbation and the boundary geometry. For explicit representations of $\Phi(t_0)$, $\Psi(t)$, $\Phi'(t)$, and $\Psi'(t)$, we refer the reader to \eqref{iph-1}, \eqref{ips-1}, and \eqref{ish}--\eqref{ihs}; see also Remark~\ref{rk-ph} in Section~\ref{sec-thm1.2}.

As a direct consequence of Theorem~\ref{mm-thm}, we further establish the boundary asymptotics of $u$ in a more explicit form.

\begin{corollary}[Boundary asymptotics]\label{cor-lm} Under the hypotheses of Theorem~\ref{mm-thm}, the following boundary asymptotics hold:
\begin{equation}\label{thm-BDL}
 \begin{aligned}
\left|\frac{u(x)- b_*}{\eps}- \frac{\gamma\mathcal{G}(x)}{\frac{\gamma f(b_*)}{\sqrt{\mathcal{A}(q(0))}} + \sqrt{2F(b_*)}}\right|
+\left|\partial_{\vec{n}}{u(x)}-\left(\frac{1}{\eps}\sqrt{\frac{2F(b_*)}{\mathcal{A}(q(0))}}- \frac{\mathcal{G}(x)}{\frac{\gamma f(b_*)}{\sqrt{\mathcal{A}(q(0))}} + \sqrt{2F(b_*)}}\right)\right|
\xrightarrow{\eps\downarrow0}0,
\end{aligned}
\end{equation}
uniformly for $x\in\partial\Omega$, where
\begin{equation}\label{inb-star}
 \mathcal{G}(x):= \frac{|\partial\Omega|\mathcal{A}'(q(0))}{|\Omega|\mathcal{A}(q(0))}
 F(b_*)\mathcal{Q}_F(b_*)
 +(N-1)\mathcal{H}_{\partial\Omega}(x)\int^{b_*}_0\sqrt{2F(s)}\,\mathrm{d}s
\end{equation}
 captures the combined nonlocal and geometric effects. Here,
$\mathcal{Q}_F(b_*)=\int_0^{b_*}\frac{q(s)-q(0)}{\sqrt{2F(s)}}\mathrm{d}s$
is defined in~\eqref{Q-F}.
\end{corollary}

When $\gamma>0$, the refined boundary asymptotic expansion of $u$ follows directly from the boundary condition~\eqref{bd} combined with \eqref{thm-BDL}. We emphasize that both $\gamma$ and the surface-to-volume ratio $\frac{|\partial\Omega|}{|\Omega|}$ play essential roles in shaping the boundary layer behavior.

In the special case $\gamma=0$ and $\mathcal{A}\equiv 1$ independent of the solution~$u$, we have $b_*=b_0$, and problem~\eqref{maineq}--\eqref{bd} reduces to a classical singularly perturbed Dirichlet problem, and \eqref{thm-BDL} simplifies to
\begin{equation*}
\partial_{\vec{n}}{u(x)}=\frac{1}{\varepsilon}\sqrt{2F(b_0)}- (N-1)\mathcal{H}_{\partial\Omega}(x)\int_0^{b_0}\sqrt{\frac{F(s)}{F(b_0)}}\,\mathrm{d}s + o_\varepsilon(1).
\end{equation*}
Such boundary asymptotics can be derived via a blow-up argument; see, for instance, \cite[Theorem~1]{s2004}. The present work, however, is concerned with a more general framework involving nonlocal effect and Robin boundary conditions, which lead to significantly richer boundary layer structures.

\subsection*{Organization of the paper} The remainder of the paper is organized as follows. 
In Section~\ref{prop1-sec}, we introduce the basic properties of $v_{\theta,\varepsilon}$ (see Section~\ref{sch-52}) and then establish Proposition~\ref{existence} (see Section~\ref{s2-pf1.1}). 
Section~\ref{sec-thm1.2} is devoted to the proof of Theorem~\ref{thm-nonlocal}, while Section~\ref{sec-thm1.3} completes the proofs of Theorem~\ref{mm-thm} and Corollary~\ref{cor-lm}. 
Finally, the appendix (Section~\ref{sec-ap}) contains the proofs of Lemmas~\ref{lem-W}, \ref{lem-PhiPsi}, and \ref{qw-lem}, which provide the key properties of $W$, $\Phi$, and $\Psi$ used throughout the analysis (see Sections~\ref{subap1}--\ref{subap3}).

\section{\bf Proof of Proposition~\ref{existence}}\label{prop1-sec}

We begin by establishing the existence and uniqueness of classical solutions to \eqref{vy-eq} for each ${\theta}>0$ and $\eps>0$. Without loss of generality, we focus on the cases $b_0 \geq 0$ and $\gamma > 0$. Consider the corresponding energy functional\begin{equation*}E_{{\theta},\eps} (V):= \int_\Omega\left( \frac{{\theta}^2}{2}|\nabla V|^2 + F(V)\right) \dxx+\frac{{\theta}^2}{2\gamma\eps} \int_{\partial\Omega} (V-b_0)^2 ,\mathrm{d}\sigma_x,\qquad V\in\mathrm{H}^1(\Omega).
\end{equation*}
By assumption \eqref{assume-f}, the strict convexity of $F$ implies that $E_{\theta,\eps}$ is a strictly convex functional on $\mathrm{H}^1(\Omega)$, which directly guarantees the uniqueness of any minimizer.

Furthermore, since $F$ satisfies the quadratic growth condition $F(s) \geq \frac{1}{2}(\inf_{\mathbb{R}} f') s^2$ and is strictly convex, the functional $E_{\theta,\eps}$ is both coercive and weakly lower semicontinuous in $\mathrm{H}^1(\Omega)$. The compactness of the trace operator $\Gamma: \mathrm{H}^1(\Omega) \to \mathrm{L}^2(\partial\Omega)$ (via the Kondrachov embedding theorem) further ensures that the boundary term is weakly continuous in $\mathrm{H}^1(\Omega)$. Consequently, by the direct method in the calculus of variations, $E_{\theta,\eps}$ attains a unique minimizer $v_{\theta,\eps} \in \mathrm{H}^1(\Omega)$, which is a weak solution to \eqref{vy-eq}.

To prove $0 \leq v_{{\theta},\eps} \leq b_0$, we employ a truncation argument. Define $$w(x) := \min\left\{\max\left\{v_{{\theta},\eps}(x), 0\right\}, b_0\right\},~x\in\overline{\Omega}.$$ 
By construction, $0 \leq w \leq b_0$ and $|\nabla w| \leq |\nabla v_{{\theta},\eps}|$ almost everywhere in $\Omega$.
By~\eqref{assume-f} and \eqref{big-F}, we have  $f(0)= 0 \leq f(b_0)$ and $F(w) \leq F(v_{{\theta},\eps})$.  Regarding the boundary term, since $b_0 \geq 0$ and noting that $w(x)=0 \neq b_0$ occurs only when $v_{{\theta},\eps}(x) \leq 0$, we have $(w-b_0)^2 \leq (v_{{\theta},\eps}-b_0)^2$ on $\partial\Omega$.  Consequently, $E_{{\theta},\eps}(w) \leq E_{{\theta},\eps}(v_{{\theta},\eps})$. By the uniqueness of the minimizer, we obtain
\begin{equation*}
\min\left\{\max\left\{v_{{\theta},\eps}(x), 0\right\}, b_0\right\} = v_{{\theta},\eps} \quad\text{in}\quad\mathrm{H}^1(\Omega),
\end{equation*}
implying $0 \leq v_{{\theta},\eps} \leq b_0$ almost everywhere in $\Omega$.

Applying the standard elliptic regularity theory (cf. \cite[Theorem 6.30]{L2013} and \cite{L1988,N}), we have $v_{{\theta},\eps}\in \mathrm{W}^{2,p}(\Omega)$ for any $p>1$, which implies $v_{{\theta},\eps}\in\mathrm{C}^{\alpha}(\overline{\Omega})$ for some $\alpha\in(0,1)$. Furthermore, the Schauder approach (cf. \cite[Theorems 6.31]{GT2001}) ensures that $v_{{\theta},\eps}\in\mathrm{C}^{1,\alpha}(\overline{\Omega})\cap\mathrm{C}^{2,\alpha}({\Omega})$ is a classical solution. The uniqueness of such a classical solution is guaranteed by the fact that $\gamma >0$ and $f$ is strictly increasing on $\mathbb{R}$.

\subsection{Basic properties of $v_{{\theta},\eps}$ }\label{sch-52}

 Due to the uniqueness of solutions to~\eqref{vy-eq}, the condition $b_0=0$ implies that $v_{\theta,\varepsilon}\equiv 0$ in $\overline{\Omega}$. Therefore, without loss of generality, we assume \eqref{a-0} and restrict attention to the nontrivial case. Under this assumption, since $0 \le v_{\theta,\varepsilon} \le b_0$ and $f$ satisfies \eqref{assume-f}, it follows that $f'(v_{\theta,\varepsilon}) \ge \min_{[0,b_0]} f'$. Consequently, equation~\eqref{vy-eq} yields
$\theta\Delta v_{{\theta},\eps}\geq (\min_{[0,b_0]}f')v_{{\theta},\eps}$ in $\Omega$. The following Lemma~\ref{intdecay} for interior estimate of $v_{\theta,\varepsilon}$ then follows from the argument in~\cite[Lemma~3.1]{LWY2020}, and we omit the proof.

\begin{lemma}\label{intdecay}
Under the same hypotheses as in Proposition~\ref{existence}, we assume \eqref{a-0}. Then, for ${\theta}\in(0,1)$ and $\eps>0$, there exists a positive constant $C^*$ independent of ${\theta}$ and $\eps$ such that
\begin{align}\label{int-est-v}
    0\leq v_{{\theta},\eps}(x)\leq{b_0}\exp\left(-\frac{C^*}{{\theta}}\mathrm{dist}(x,{\partial\Omega})\right),\,\,\mathrm{for}\,\,x\in\Omega.
\end{align}
\end{lemma}

The following property plays a crucial role in dealing with \eqref{u-to-W}.
\begin{lemma}\label{lem-W}
Assume \eqref{assume-f} and \eqref{a-0}. Then for $\gamma\geq0$, \eqref{ode-W} has a unique solution $W$. Moreover,
\begin{align}\label{w1w}
    W(0)=b_*>0, \ \ W'(0)=-\sqrt{\frac{2F(b_*)}{\mathcal{A}(q(0))}}, \ \ W'(t)=-\sqrt{\frac{2F(W(t))}{\mathcal{A}(q(0))}}< 0,
\end{align}
and 
\begin{align}\label{w3w}
    0< W(t) \le b_* \mathrm{e}^{-M_0t},\quad|W'(t)|\leq\sqrt{\frac{2b_*f(b_*)}{\mathcal{A}(q(0))}}\mathrm{e}^{-\frac{M_0}{2}t},
\end{align}
for $t\geq0$, where $b_*>0$ was defined by \eqref{b0} and $M_0=\textstyle\sqrt{\frac{1}{\mathcal{A}(q(0))}\min\limits_{[0,b_*]}f'}>0$.
\end{lemma}
For the sake of conciseness in the main text, the proof of Lemma~\ref{lem-W} is placed in the Appendix.
 
 \subsection{ Completion of the proof of Proposition~\ref{existence}}\label{s2-pf1.1}

For a fixed $\eps > 0$, we observe that if $\boldsymbol{\mathsf{M}}_{\eps}({\theta}) = 0$, then $v_{{\theta},\eps}$ is a solution to \eqref{maineq} with the boundary condition \eqref{bd}, where $\boldsymbol{\mathsf{M}}_{\eps}$ is defined in \eqref{vy-map}. Conversely, given a solution $u$ to \eqref{maineq}--\eqref{bd}, we define
\begin{equation*}
\widetilde{{\theta}}(\eps, u) := \eps \sqrt{\mathcal{A} \left( \fint_\Omega {q}(u) \,\mathrm{d}x \right)}.
\end{equation*}
It then follows from the uniqueness of solutions to \eqref{vy-eq} that
\begin{equation}\label{u-vye}
u \equiv v_{\widetilde{{\theta}}(\eps, u), \eps} \quad \text{and} \quad \boldsymbol{\mathsf{M}}_{\eps}(\widetilde{{\theta}}(\eps, u)) = 0.
\end{equation}
Consequently, $u$ is a solution to \eqref{maineq}--\eqref{bd} if and only if $u = v_{{\theta},\eps}$ for some $\theta > 0$ satisfying the algebraic relation $\boldsymbol{\mathsf{M}}_{\eps}({\theta}) = 0$.

To complete the proof of Proposition~\ref{existence}, it remains to show that, for sufficiently small $\varepsilon>0$, the equation $\boldsymbol{\mathsf{M}}_{\varepsilon}(\theta)=0$ admits a unique solution $\theta=\theta(\varepsilon)$. More precisely, we establish the following result.

\begin{lemma}\label{lem-vy-map}
Under the same assumptions as in Proposition~\ref{existence}, suppose that \eqref{a-0} holds. Then there exists a positive constant $\boldsymbol{\mathfrak{e}^*}$, depending only on $\mathcal{A}(q(0))$, $\displaystyle \max_{I}\mathcal{A}$, $\displaystyle \max_{I}|\mathcal{A}'|$, and $\displaystyle \max_{[0,b_0]}|q'|$, such that for each $\varepsilon \in (0,\boldsymbol{\mathfrak{e}^*})$, $\boldsymbol{\mathsf{M}}_{\varepsilon}(\theta)=0$ admits a unique zero $\theta=\theta(\varepsilon)>0$. Here,
$I=\ds[\min_{[0,b_0]}q,\max_{[0,b_0]}q]$.
\end{lemma}
 \begin{proof}
 For each fixed $\varepsilon>0$, one can follow the argument in~\cite[Theorem~1.1]{T2000} to conclude that $v_{\theta,\varepsilon}$ depends continuously differentiably on the parameter $\theta>0$ in $\mathrm{H}^1(\Omega)$ (see the footnote\footnote{Consider $E(\eta,V):=E_{\theta,\varepsilon}(V): \mathbb{R}_+ \times \mathrm{H}^1(\Omega)\to \mathbb{R}$. Then $v_{\theta,\varepsilon}$ is a critical point of $E(\theta,\cdot)$. The second variation $\frac{\partial^2 E}{\partial V^2}(\theta, v_{\theta,\varepsilon})$ defines a bounded and coercive bilinear form on $\mathrm{H}^1(\Omega)$ by \eqref{assume-f} and the trace inequality, and thus induces an equivalent inner product. In particular, it is an isomorphism in $\mathcal{L}(\mathrm{H}^1(\Omega),(\mathrm{H}^1(\Omega))^*)$ by the Lax--Milgram theorem. Hence, by the implicit function theorem, $v_{\theta,\varepsilon}$ depends $C^1$-smoothly on $\theta>0$.}). We denote by $\frac{\mathrm{d}v_{\theta,\varepsilon}}{\mathrm{d}\theta}$ the derivative defined as the limit $\lim_{h\to 0}\frac{v_{\theta+h,\varepsilon}-v_{\theta,\varepsilon}}{h}$ for $\theta\in(0,\infty)$. Consequently, for each $\varepsilon>0$, $\boldsymbol{\mathsf{M}}_{\varepsilon}(\theta)$ is a continuously differentiable function of $\theta\in(0,\infty)$.

 We now fix $\varepsilon>0$ sufficiently small such that $\varepsilon<\min_{I}\frac1{\sqrt{\mathcal{A}}}$. Then $\boldsymbol{\mathsf{M}}_{\eps}(1,\eps)=1-\eps^2{\mathcal{A}\left(\fint_\Omega {q}(v_{1,\eps}) \dxx\right)}>0$ since $0\leq v_{1,\eps}(x)\leq b_0$ on $\overline{\Omega}$. On the other hand, by \eqref{int-est-v} we have $\lim_{{\theta}\downarrow0}\boldsymbol{\mathsf{M}}_{\eps}({\theta})=-\eps^2\mathcal{A}(q(0))<0$. Hence, there exists at least one ${\theta}={\theta}(\eps)\in(0,1)$ depending on $\eps$ such that $\boldsymbol{\mathsf{M}}_{\eps}({\theta}(\eps),\eps)=0$. Moreover, by \eqref{int-est-v} again, any such $\theta(\varepsilon)$ satisfies
\begin{align}\label{lim-eta-eps}
\lim_{\varepsilon \downarrow 0} \frac{\theta(\varepsilon)}{\varepsilon}
= \lim_{\varepsilon \downarrow 0} \sqrt{\mathcal{A}\!\left(\fint_\Omega q(v_{\theta(\varepsilon),\varepsilon}) \,\mathrm{d}x\right)}
= \sqrt{\mathcal{A}(q(0))}.
\end{align}
Consequently, $u_{\varepsilon}=v_{\theta(\varepsilon),\varepsilon}$ is a solution of \eqref{maineq} subject to the boundary condition~\eqref{bd}, and \eqref{lim-eta-eps} ensures the existence of a positive constant $\mathfrak{e}$ such that, for all $\varepsilon \in (0,\mathfrak{e})$, any corresponding $\theta(\varepsilon)$ lies in the interval  $(\frac{\eps}{2}\sqrt{\mathcal{A}(q(0))},\frac{3\eps}{2}\sqrt{\mathcal{A}(q(0))})$. This conclusion also implies
\begin{align}\label{par-m}
\boldsymbol{\mathsf{M}}_{\varepsilon}(\theta) \neq 0 
\quad \text{whenever} \quad 
0<\varepsilon<\min\left\{\min_{I}\frac1{\sqrt{\mathcal{A}}},\mathfrak{e}\right\}
\ \text{and} \ 
\theta \notin  ( \frac{\varepsilon}{2}\sqrt{\mathcal{A}(q(0))}, \frac{3\varepsilon}{2}\sqrt{\mathcal{A}(q(0))}).
\end{align}
Accordingly, it suffices to find a constant $\boldsymbol{\mathfrak{e}^*} < \min\{\min_{I}\frac1{\sqrt{\mathcal{A}}},\mathfrak{e}\}$ such that, for each fixed $\varepsilon \in (0,\boldsymbol{\mathfrak{e}^*})$, the function $\boldsymbol{\mathsf{M}}_{\varepsilon}(\theta)$ is strictly monotone on $\left( \tfrac{\varepsilon}{2}\sqrt{\mathcal{A}(q(0))}, \tfrac{3\varepsilon}{2}\sqrt{\mathcal{A}(q(0))} \right)$. This would imply that, for each $\varepsilon \in (0,\boldsymbol{\mathfrak{e}^*})$, the equation $\boldsymbol{\mathsf{M}}_{\varepsilon}(\theta)=0$ admits a unique solution $\theta=\theta(\varepsilon)$.
 
 To achieve this goal, we now fix $0<\varepsilon< \min\{\min_{I}\frac1{\sqrt{\mathcal{A}}},\mathfrak{e}\}$. Differentiating $\boldsymbol{\mathsf{M}}_{\varepsilon}$ with respect to $\theta$, we obtain
\begin{align}\label{diff-m}
\frac{\mathrm{d}\boldsymbol{\mathsf{M}}_{\varepsilon}}{\mathrm{d}\theta}(\theta)
= 2\theta
- \varepsilon^2 \mathcal{A}'\!\left(\fint_{\Omega} q(v_{\theta,\varepsilon})\,\mathrm{d}x\right)
\fint_{\Omega} q'(v_{\theta,\varepsilon}) \frac{\mathrm{d}v_{\theta,\varepsilon}}{\mathrm{d}\theta}\,\mathrm{d}x.
\end{align}
To deal with the right-hand side of \eqref{diff-m}, we require the following claim.

\medskip
\noindent
\textbf{Claim 1.}
Let $\varepsilon>0$ be fixed. Then, for each $x\in\overline{\Omega}$,   $\frac{\mathrm{d}v_{{\theta},\eps}}{\mathrm{d}{\theta}}\geq0$. Moreover, it holds that
\begin{align}\label{max-v-eta}
\sup_{0<\theta<1} \sqrt{\theta} \int_{\Omega} \frac{\mathrm{d}v_{\theta,\varepsilon}}{\mathrm{d}\theta}\,\mathrm{d}x < \infty.
\end{align}
  \begin{proof}[Proof of Claim~1]
  Differentiating the equation of $v_{{\theta},\eps}$ in \eqref{vy-eq} with respect to ${\theta}$ and making appropriate manipulations, we arrive at
  \begin{align}\label{new-diff-v}
     {\theta}^2\Delta \frac{\mathrm{d}v_{{\theta},\eps}}{\mathrm{d}{\theta}}=f'(v_{{\theta},\eps})\frac{\mathrm{d}v_{{\theta},\eps}}{\mathrm{d}{\theta}}-\frac{2f(v_{{\theta},\eps})}{{\theta}}\,\,\mathrm{in}\,\,\Omega,
  \end{align}
  and 
  \begin{align}\label{diff-bd}
   \frac{\mathrm{d}v_{{\theta},\eps}}{\mathrm{d}{\theta}}+\gamma\eps\partial_{\vec{n}}\left(\frac{\mathrm{d}v_{{\theta},\eps}}{\mathrm{d}{\theta}}\right)=0\,\,\mathrm{on}\,\,\partial\Omega.   
  \end{align}
  Since $\gamma\geq0$ and $f'(v_{{\theta},\eps})\geq\ds\min_{[0,b_0]}f'>0$, we can apply the maximum principle to \eqref{new-diff-v}--\eqref{diff-bd} and thus obtain
  \begin{align}\label{diff-v}
    \frac{\mathrm{d}v_{{\theta},\eps}}{\mathrm{d}{\theta}}\geq0\,\,\mathrm{in}\,\,\overline{\Omega}.
  \end{align}

  Now we consider ${\theta}\in(0,1)$. Note that $\gamma\geq0$. Multiplying  \eqref{new-diff-v} by $\frac{\mathrm{d}v_{{\theta},\eps}}{\mathrm{d}{\theta}}$,  integrating the expression over $\Omega$ and using \eqref{diff-bd}, one may check that
  \begin{align*}   \fint_{\Omega}f'(v_{{\theta},\eps})\left(\frac{\mathrm{d}v_{{\theta},\eps}}{\mathrm{d}{\theta}}\right)^2\dxx=&\frac{2}{{\theta}}\fint_{\Omega}f(v_{{\theta},\eps})\frac{\mathrm{d}v_{{\theta},\eps}}{\mathrm{d}{\theta}}\dxx-{\theta}^2\fint_{\Omega}\left|\nabla\frac{\mathrm{d}v_{{\theta},\eps}}{\mathrm{d}{\theta}}\right|^2\dxx\\
      &+\frac{{\theta}^2}{|\Omega|}\int_{\partial\Omega}\frac{\mathrm{d}v_{{\theta},\eps}}{\mathrm{d}{\theta}}\partial_{\vec{n}}\left(\frac{\mathrm{d}v_{{\theta},\eps}}{\mathrm{d}{\theta}}\right)\mathrm{d}\sigma_x\\
      \leq&\frac{1}{{\theta}^2C_{f'}}\fint_{\Omega}(f(v_{{\theta},\eps}))^2\dxx+{C_{f'}}\fint_{\Omega}\left(\frac{\mathrm{d}v_{{\theta},\eps}}{\mathrm{d}{\theta}}\right)^2\dxx,
  \end{align*}
  where $C_{f'}:=\frac{1}{2}\min\limits_{[0,b_0]}f'>0$. Here we have used \eqref{diff-bd} and \eqref{diff-v} to verify $\partial_{\vec{n}}\left(\frac{\mathrm{d}v_{{\theta},\eps}}{\mathrm{d}{\theta}}\right)\leq0$ on $\partial\Omega$. The last estimate is obtained by the elementary inequality $\frac{2}{{\theta}}f(v_{{\theta},\eps})\frac{\mathrm{d}v_{{\theta},\eps}}{\mathrm{d}{\theta}}\leq\frac{1}{{\theta}^2C_{f'}}(f(v_{{\theta},\eps}))^2+C_{f'}(\frac{\mathrm{d}v_{{\theta},\eps}}{\mathrm{d}{\theta}})^2$.  Since $\min_{\overline{\Omega}}f'(v_{{\theta},\eps})\geq2C_{f'}$, we arrive at
  \begin{align*}
   \fint_{\Omega}\left(\frac{\mathrm{d}v_{{\theta},\eps}}{\mathrm{d}{\theta}}\right)^2\dxx\leq\frac{1}{{\theta}^2C_{f'}^2}\fint_{\Omega}(f(v_{{\theta},\eps}))^2\dxx.   
  \end{align*}
   Combining this estimate with \eqref{int-est-v} and \eqref{diff-v} yields that, for ${\theta}\in(0,1)$,
  \begin{equation}\label{int-v}
0\leq\fint_{\Omega}\frac{\mathrm{d}v_{{\theta},\eps}}{\mathrm{d}{\theta}}\dxx
 \leq  \frac{1}{{\theta}C_{f'}}\sqrt{\fint_{\Omega}(f(v_{{\theta},\eps}))^2\dxx}
   \leq\frac{1}{{\theta}C_{f'}}\max_{[0,b_0]}f'\sqrt{\fint_{\Omega}v_{{\theta},\eps}^2\dxx}\leq\frac{{ C}^{*}_f}{\sqrt{{\theta}}},
  \end{equation}
  where ${C}^{*}_f$ is a positive constant independent of ${\theta}$.  The last inequality of \eqref{int-v} comes from the estimate $\fint_{\Omega}v_{{\theta},\eps}^2\dxx=O({\theta})$ which can be obtained by \eqref{assume-f} and \eqref{int-est-v}. Therefore, we arrive at \eqref{max-v-eta} and complete the proof.
  \end{proof}

Next, we define
\begin{align}\label{math-e}
\boldsymbol{\mathfrak{e}^*}
=
\min\left\{
\mathfrak{e},\,\min_{I}
\frac{1}{\sqrt{\mathcal{A}}},\,
\frac{\big(\mathcal{A}(q(0))\big)^{3/2}}{2\big(C_f^* \max_I|\mathcal{A}'| \max_{[0,b_0]}|q'|\big)^2}
\right\},
\end{align}
where $\mathfrak{e}$ is given in \eqref{par-m}. The last term in \eqref{math-e} is chosen to guarantee the strict monotonicity of $\boldsymbol{\mathsf{M}}_{\varepsilon}$ on $\left(\tfrac{\varepsilon}{2}\sqrt{\mathcal{A}(q(0))}, \tfrac{3\varepsilon}{2}\sqrt{\mathcal{A}(q(0))}\right)$.

Indeed, for any fixed $\varepsilon \in (0,\boldsymbol{\mathfrak{e}^*})$ and $\theta \in \left(\tfrac{\varepsilon}{2}\sqrt{\mathcal{A}(q(0))}, \tfrac{3\varepsilon}{2}\sqrt{\mathcal{A}(q(0))}\right)$, combining \eqref{diff-m} with \eqref{int-v} yields
\begin{align*}
\frac{\mathrm{d}\boldsymbol{\mathsf{M}}_{\varepsilon}}{\mathrm{d}\theta}(\theta)
&\ge
2\theta
-
\varepsilon^2 \max_I |\mathcal{A}'| \max_{[0,b_0]} |q'|
\fint_{\Omega} \frac{\mathrm{d}v_{\theta,\varepsilon}}{\mathrm{d}\theta}\,\mathrm{d}x \\
&\ge
\frac{1}{\sqrt{\theta}}
\left(
2\theta\sqrt{\theta}
-
\varepsilon^2 C_f^* \max_I |\mathcal{A}'| \max_{[0,b_0]} |q'|
\right).
\end{align*}
By $\theta > \tfrac{\varepsilon}{2}\sqrt{\mathcal{A}(q(0))}$ and  \eqref{math-e}, we further obtain
\begin{align*}
\frac{\mathrm{d}\boldsymbol{\mathsf{M}}_{\varepsilon}}{\mathrm{d}\theta}(\theta)
\geq
\frac{\varepsilon\sqrt{\varepsilon}}{\sqrt{\theta}}
\left(
\frac{\big(\mathcal{A}(q(0))\big)^{3/4}}{\sqrt{2}}
-
\sqrt{\varepsilon}\, C_f^* \max_I |\mathcal{A}'| \max_{[0,b_0]} |q'|
\right)
> 0.
\end{align*}
Therefore, $\boldsymbol{\mathsf{M}}_{\varepsilon}$ is strictly increasing on 
$\left(\tfrac{\varepsilon}{2}\sqrt{\mathcal{A}(q(0))}, \tfrac{3\varepsilon}{2}\sqrt{\mathcal{A}(q(0))}\right)$ for all $\varepsilon \in (0,\boldsymbol{\mathfrak{e}^*})$.

Together with \eqref{lim-eta-eps} and \eqref{par-m}, this implies that, for each $\varepsilon \in (0,\boldsymbol{\mathfrak{e}^*})$, the equation $\boldsymbol{\mathsf{M}}_{\varepsilon}(\theta)=0$ admits a unique solution $\theta=\theta(\varepsilon)$.  This completes the proof of Lemma~\ref{lem-vy-map}.
\end{proof}

We are now ready to complete the proof of Proposition~\ref{existence}. 
Suppose, to the contrary, that there exists $\eps\in(0,\boldsymbol{{\mathfrak{e}}^*})$ such that the problem \eqref{maineq} with the boundary condition~\eqref{bd} admits at least two distinct solutions $u_1$ and $u_2$. 
Define
\begin{align*}
\widetilde{{\theta}}_i(\eps):=\eps\sqrt{\mathcal{A}\left(\fint_\Omega {q}(u_i)\,\dxx\right)}, \quad i=1,2.  
\end{align*}
Then, by \eqref{u-vye}, it follows that
\begin{align*}
u_i \equiv v_{\widetilde{{\theta}}_i(\eps),\eps}
\in \mathrm{C}^{1,\alpha}(\overline{\Omega}) \cap \mathrm{C}^{2,\alpha}({\Omega}),
\quad \text{and} \quad 
\boldsymbol{\mathsf{M}}_{\eps}(\widetilde{{\theta}}_i(\eps))=0.
\end{align*}
Since $u_1 \not\equiv u_2$, we have $\widetilde{{\theta}}_1(\eps)\neq \widetilde{{\theta}}_2(\eps)$. 
Thus, the equation $\boldsymbol{\mathsf{M}}_{\eps}({\theta})=0$ admits at least two distinct solutions, which contradicts Lemma~\ref{lem-vy-map}. 
Therefore, the solution to \eqref{maineq} with the boundary condition~\eqref{bd} is unique for all $\eps\in(0,\boldsymbol{{\mathfrak{e}}^*})$.

The estimate \eqref{int-est-u} follows immediately from \eqref{int-est-v} and \eqref{lim-eta-eps}, where $C_*>0$ is a constant independent of $\eps$, and behaves like $\frac{C^*}{\sqrt{\mathcal{A}(q(0))}}$ as $0<\eps \ll1$. 
It remains to establish \eqref{u-to-W}.

 Let $\delta(x):=\mathrm{dist}(x,\partial\Omega)$ (see the geometric property~(G-ii)). Then, by \eqref{maineq} and \eqref{ode-W} one may check that
  \begin{equation}\label{322-2021}
       \begin{aligned}
      \eps^2\Delta\Big(u(x)-W(\frac{\delta(x)}{\eps})\Big)
      =&\,\frac{f(u(x))}{\mathcal{A}(\fint_{\Omega}q(u(y))\,\mathrm{d}y)}-\frac{f(W(\frac{\delta(x)}{\eps}))}{\mathcal{A}(q(0))}-\eps{W'(\frac{\delta(x)}{\eps})}\Delta\delta(x)\\
      =&\,\frac{1}{\mathcal{A}(q(0))}\left(f(u(x))-f(W(\frac{\delta(x)}{\eps}))\right)+\mathcal{U}_{\eps}(x),
  \end{aligned}
  \end{equation}
  where
  \begin{align}\label{bbU}
   \mathcal{U}_{\eps}(x):=f(u(x))\left(\frac{1}{\mathcal{A}(\fint_{\Omega}q(u(y))\,\mathrm{d}y)}-\frac{1}{\mathcal{A}(q(0))}\right)-\eps{W'(\frac{\delta(x)}{\eps})}\Delta\delta(x).   
  \end{align}
   Here we have used the fact $|\nabla\delta(x)|=1$. 
   
To deal with $\mathcal{U}_{\eps}$, we first notice that \eqref{int-est-u} implies
  \begin{equation}\label{ae}
\mathcal{A}\left(\fint_\Omega q(u(y))\,\mathrm{d}y\right)=\mathcal{A}(q(0))+O(\eps)\quad\mathrm{as}\quad\eps\in(0,\boldsymbol{{\mathfrak{e}}^*}).
\end{equation}
Hence, for \eqref{bbU}, by \eqref{assume-f} and \eqref{ae}, we have the estimate
  \begin{align}\label{322-2021-0}
     |\mathcal{U}_{\eps}(x)|\leq\max_{[0,b_0]}f \left|\frac{1}{\mathcal{A}(q(0))+O(\eps)}-\frac{1}{\mathcal{A}(q(0))}\right|+\eps\sup_{[0,\infty)}|W'|\max_{\overline{\Omega}}|\Delta\delta|\leq{C}_3\eps,
  \end{align}
  for some positive constant $C_3$ independent of $\eps$. Since $u,W\in[0,b_0]$ (see  Lemma~\ref{lem-W}), \eqref{322-2021} and \eqref{322-2021-0} imply
  \begin{equation}\label{u-W-dieq}
   \begin{aligned}
      \eps^2\Delta\Big(u(x)-W(\frac{\delta(x)}{\eps})\Big)^2\geq&\,2\eps^2\Big(u(x)-W(\frac{\delta(x)}{\eps})\Big)\Delta\Big(u(x)-W(\frac{\delta(x)}{\eps})\Big)\\ \geq&\,\frac{2\min_{[0,b_0]}f'}{\mathcal{A}(q(0))}\Big(u(x)-W(\frac{\delta(x)}{\eps})\Big)^2-2C_3\eps\Big|u(x)-W(\frac{\delta(x)}{\eps})\Big|,
  \end{aligned}   
  \end{equation}
 where we recall again that  $\min_{[0,b_0]}f'>0$.
  
  On the other hand, since $\partial_{\vec{n}}\delta(x)=-1$ on $\partial\Omega$, by \eqref{bd} and \eqref{ode-W}, one obtains
  \begin{align}\label{u-w-bd}
      \Big(u(x)-W(\frac{\delta(x)}{\eps})\Big)^2+\frac{\gamma\eps}{2}\partial_{\vec{n}}\Big(u(x)-W(\frac{\delta(x)}{\eps})\Big)^2=0,\,\,x\in\partial\Omega.
  \end{align}
Since $\gamma \geq 0$, it follows from \eqref{322-2021}, \eqref{u-W-dieq} and \eqref{u-w-bd} that $(u(x)-W(\frac{\delta(x)}{\eps}))^2$ attains its (positive) maximum at some interior point $\widetilde{x}\in\Omega$. Consequently,
\begin{equation*}
 \left|u(x)-W(\frac{\delta(x)}{\eps})\right|
\leq 
\left|u(\widetilde{x})-W(\frac{\delta(\widetilde{x})}{\eps})\right|
\leq\frac{
C_3\eps}{\min_{[0,b_0]}f'},
\quad \forall\,x\in\overline{\Omega},  
\end{equation*}
which yields \eqref{u-to-W}. 
Combining the above results, we complete the proof of Proposition~\ref{existence}.
 
\begin{remark} It is worth highlighting that when $\gamma > 0$ in the boundary condition~\eqref{diff-bd}, the estimate \eqref{max-v-eta} can be further refined to 
\begin{align}\label{rk-diff-v}
\sup_{0<{\theta}<1} \int_{\Omega} \frac{\mathrm{d}v_{{\theta},\eps}}{\mathrm{d}{\theta}}\,\mathrm{d}x < \infty.
\end{align}
To see this, we integrate \eqref{new-diff-v} over $\Omega$. By invoking \eqref{diff-bd} and \eqref{diff-v}, we obtain
{\footnotesize\begin{align*}
\int_\Omega\frac{\mathrm{d}v_{{\theta},\eps}}{\mathrm{d}{\theta}}\dxx\leq\frac{1}{2C_{f'}}\int_{\Omega}f'(v_{{\theta},\eps})\frac{\mathrm{d}v_{{\theta},\eps}}{\mathrm{d}{\theta}}\dxx=\frac{1}{2C_{f'}}\left(\frac{2}{{\theta}}\int_{\Omega}f(v_{{\theta},\eps})\dxx-\frac{{\theta}^2}{\gamma\eps}\int_{\partial\Omega}\frac{\mathrm{d}v_{{\theta},\eps}}{\mathrm{d}{\theta}}\,\mathrm{d}\sigma_x\right)\leq\frac{1}{{\theta}C_{f'}}\int_{\Omega}f(v_{{\theta},\eps})\dxx.
\end{align*}}
Combining this with \eqref{assume-f} and \eqref{int-est-v} yields the desired estimate \eqref{rk-diff-v}.
\end{remark}
\begin{remark}[A sufficient condition for uniqueness of \eqref{maineq}--\eqref{bd} with special $\mathcal{A}$ and $q$]\label{rk-prop1}
Lemma~\ref{lem-vy-map} establishes the uniqueness of ${\theta}(\eps)$ for sufficiently small $\eps$ under general assumptions on $\mathcal{A}$ and $q$. In contrast, if one of the conditions in \eqref{ass-i-ii} is satisfied, the uniqueness of ${\theta}(\eps)$ holds for all $\eps>0$. This follows directly from \eqref{diff-m} and \eqref{diff-v}. Consequently, under either of the assumptions in \eqref{ass-i-ii}, the problem \eqref{maineq} with the boundary condition~\eqref{bd} admits a unique solution $u$ for all $\eps>0$.
\end{remark}

\section{\bf Proof of Theorem \ref{thm-nonlocal}}\label{sec-thm1.2} 

In what follows, we set
\begin{align}\label{b-eps}
    B_\e:= \fint_\Omega  q(u)\dxx.
\end{align}
It follows from \eqref{int-est-u} that 
$\limsup_{\eps\downarrow0}\frac1{\eps}|B_\e-q(0)|<\infty.$
To prove Theorem~\ref{thm-nonlocal}, it suffices to derive a refined asymptotic expansion of $\mathcal{A}(B_\varepsilon)$ in \eqref{ae} as $\varepsilon \downarrow 0$. We begin by establishing the following property for the solutions of \eqref{Phi-new} and \eqref{Psi-new}.
\begin{lemma}\label{lem-PhiPsi}
Assume \eqref{q-A}--\eqref{assume-f}. Then
 equations \eqref{Phi-new} and \eqref{Psi-new} have unique solutions. Moreover, we have that:
 \begin{itemize}
     \item[(i)]  $\Phi<0$ in $[0,\infty)$ and 
     \begin{align}
         \Phi'(t)+\frac{f(W(t))\Phi(t)}{\sqrt{2\mathcal{A}(q(0))F(W(t))}}&\,=-\sqrt{\frac{\mathcal{A}(q(0))F(W(t))}{2}},\quad\,for\,\,t\geq0,\label{iph-1}\\
         \Phi(0)=-\frac{\gamma \mathcal{A}(q(0)) F(b_*)}{\gamma f(b_*)+\sqrt{2\mathcal{A}(q(0)) F(b_*)}},& \quad \Phi'(0)=-\frac{\mathcal{A}(q(0))F(b_*)}{\gamma f(b_*)+\sqrt{2\mathcal{A}(q(0))F(b_*)}}.\label{iph-2}
     \end{align}
     \item[(ii)] $\Psi>0$ in $[0,\infty)$ and 
      \begin{align}
         \Psi'(t) +\frac{f(W(t))\Psi(t)}{\sqrt{2\mathcal{A}(q(0))F(W(t))}}&\,=\sqrt{\frac{\mathcal{A}(q(0))}{2F(W(t))}} G(W(t)),\quad\,for\,\,t\geq0,\label{ips-1}\\
         \Psi(0)=\frac{\gamma\mathcal{A}(q(0))G(b_*)}{\gamma f(b_*) + \sqrt{2\mathcal{A}(q(0))F(b_*)}},&\, \quad \Psi'(0)=\frac{\mathcal{A}(q(0))G(b_*)}{\gamma f(b_*)+\sqrt{2\mathcal{A}(q(0))F(b_*)}},\label{ips-2}
     \end{align}
     where $G(t)=\int^{t}_0\sqrt{\frac{2F(s)}{\mathcal{A}(q(0))}}\mathrm{d}s.$
\item[(iii)] There exist $\widetilde{C}$ and $\widetilde{M}$ independent of $\eps$ such that
\begin{align}\label{cw-exp}
  |\Phi(t)|+|\Phi'(t)| + |\Psi(t)|+|\Psi'(t)|\leq\widetilde{C}\mathrm{e}^{-\widetilde{M}t},\quad\,t\in[0,\infty).
\end{align}
 \end{itemize}
\end{lemma}
For the sake of brevity, the proof of Lemma~\ref{lem-PhiPsi} is deferred to Section~\ref{subap2}.

To capture the influence of boundary curvature on $u_\varepsilon$, we derive a refined asymptotic expansion beyond \eqref{u-to-W}. Let $\delta(x)=\mathrm{dist}(x,\partial\Omega)$, and denote by $\mathcal{H}_{\Gamma_{\delta(x)}}$ and $\mathcal{H}_{\partial\Omega}$ the mean curvatures of the parallel surface $\Gamma_{\delta(x)}$ and the boundary $\partial\Omega$, respectively, as introduced in (G-ii). We also make use of \eqref{ode-W}, \eqref{Phi-new}, \eqref{Psi-new}, and the definition of the nonlocal coefficient $B_\varepsilon$ in \eqref{b-eps}. The following estimate holds.

\begin{lemma}\label{lem2}
Let $\Omega_{d_*}$ be the subdomain of $\Omega$ defined via $d_*$ in \eqref{d-star}. Then, for each $x \in \overline{\Omega_{d_*}}$, there exists a unique $\sigma(x)\in\partial\Omega$ such that \eqref{x-to=sigma} holds. Moreover, for $0<\varepsilon\ll1$, there exists a constant $C_{d_*}>0$, independent of $\varepsilon$, such that
\begin{align}\label{moon-1}
\max_{x\in\overline{\Omega_{d_*}}}\left|u(x)- W(\frac{\delta(x)}{\e}) -\eps \boldsymbol{\Xi}_\e(\delta(x),\sigma(x)) \right| \le C_{d_*} \eps^2,
\end{align}
where
\begin{equation}\label{Xi0} \boldsymbol{\Xi}_\e(\delta(x),\sigma(x)):=
\frac{\mathcal{A}(q(0))-\mathcal{A}(B_\e)}{\eps(\mathcal{A}(q(0)))^2}
\Phi(\frac{\delta(x)}{\e})
+  (N-1)\mathcal{H}_{\Gamma_{\delta(x)}}(\sigma(x))
\Psi(\frac{\delta(x)}{\e}).
\end{equation}
In particular, if $x_\varepsilon \in \overline{\Omega_{d_*}}$ satisfies $\delta(x_\varepsilon)=O(\eps)$ as $0<\eps\ll1$, then the refined pointwise asymptotic expansion of $u(x_\eps)$ is given by
{\small\begin{align}\label{udstar}
    u(x_\eps)= W(\frac{\delta(x_\eps)}{\eps}) + \frac{\mathcal{A}(q(0))-\mathcal{A}(B_\e)}{\left(\mathcal{A}(q(0))\right)^2}\Phi(\frac{\delta(x_\eps)}{\eps})
+\eps (N-1) \mathcal{H}_{\partial\Omega}(\sigma(x_\eps))\Psi(\frac{\delta(x_\eps)}{\eps})+O(\eps^2).
\end{align}}
\end{lemma}
\begin{proof}
For $x \in \Omega_{d_*}$, it follows from \eqref{Laplace}, \eqref{Phi-new}, and \eqref{Psi-new} that 
\begin{align}
    \e^2 \Delta \Phi(\frac{\delta(x)}{\e})=&\,\frac{f'(W(\frac{\delta(x)}{\e}))}{\mathcal{A}(q(0))}\Phi(\frac{\delta(x)}{\e})+f(W(\frac{\delta(x)}{\e})) - \e (N-1)\mathcal{H}_{\Gamma_{\delta(x)}}(\sigma(x)) \Phi'(\frac{\delta(x)}{\e}),\label{Phi-2}\\
\e^2 \Delta \Psi(\frac{\delta(x)}{\e})=&\,\frac{f'(W(\frac{\delta(x)}{\e}))}{\mathcal{A}(q(0))}\Psi(\frac{\delta(x)}{\e})+W'(\frac{\delta(x)}{\e}) - \e (N-1)\mathcal{H}_{\Gamma_{\delta(x)}}(\sigma(x)) \Psi'(\frac{\delta(x)}{\e}).\label{Psi-2}
\end{align}
Here we have used the fact that both $\Phi\!\left(\frac{\delta(x)}{\e}\right)$ and $\Psi\!\left(\frac{\delta(x)}{\e}\right)$ are constant along each level surface $\Gamma_{\delta}$, and hence their Beltrami--Laplacians vanish, i.e., $\Delta_{\Gamma_{\delta}}\Phi(\frac{\delta(x)}{\e})=0$ and $\Delta_{\Gamma_{\delta}}\Psi(\frac{\delta(x)}{\e})=0$.

We now outline a direct approach to analyze the asymptotic behavior of $u_\e$. 
Recalling the definition of $B_\e$ in \eqref{b-eps}, it follows from \eqref{Laplace} and \eqref{322-2021}--\eqref{bbU} that, for $x \in \Omega_{d_*}$,
\begin{align}\label{w-de}
 \eps^2\Delta\left(u(x)-W(\frac{\delta(x)}{\eps})\right)=&\,\frac{f(u(x))}{\mathcal{A}(B_{\eps})}-\frac{f(W(\frac{\delta(x)}{\eps}))}{\mathcal{A}(q(0))}+\eps (N-1)\mathcal{H}_{\Gamma_{\delta(x)}}(\sigma(x))W'(\frac{\delta(x)}{\eps})
\end{align}
where we have used the identity $\Delta\delta(x)=-(N-1)\mathcal{H}_{\Gamma_{\delta(x)}}(\sigma(x))$ in $\Omega_{d_*}$. We observe that the terms $f\!\left(W\!\left(\frac{\delta(x)}{\eps}\right)\right)$ and $W'\!\left(\frac{\delta(x)}{\eps}\right)$ also appear in \eqref{Phi-2} and \eqref{Psi-2}. 
Motivated by this structure, and by comparing the corresponding terms in \eqref{Phi-2}--\eqref{w-de}, we set
\begin{align}\label{new-u}
   U_{\eps}(x):= u(x)-W(\frac{\delta(x)}{\eps})-\left(\frac{1}{\mathcal{A}(B_\e)}-\frac{1}{\mathcal{A}(q(0))}\right)\Phi(\frac{\delta(x)}{\e})- \e (N-1) \mathcal{H}_{\Gamma_{\delta(x)}}(\sigma(x))\Psi(\frac{\delta(x)}{\e}).
\end{align}
Then for $x\in\Omega_{d_*}$, by \eqref{Phi-2}--\eqref{new-u}, one may check that
\begin{align*}
      \eps^2\Delta U_{\eps}(x)
     =&\,\eps^2\Delta\left(u(x)-W(\frac{\delta(x)}{\eps})\right)-\eps^2\left(\frac{1}{\mathcal{A}(B_\e)}-\frac{1}{\mathcal{A}(q(0))}\right)\Delta\Phi(\frac{\delta(x)}{\e})\notag\\
     &\,-\eps^3(N-1)\Delta\left(\mathcal{H}_{\Gamma_{\delta(x)}}(\sigma(x))\Psi(\frac{\delta(x)}{\e})\right)\notag\\
     =&
     \,\frac{f(u(x))}{\mathcal{A}(B_\e)}-\frac{f(W(\frac{\delta(x)}{\eps}))}{\mathcal{A}(q(0))} +\eps (N-1)\mathcal{H}_{\Gamma_{\delta(x)}}(\sigma(x))W'(\frac{\delta(x)}{\eps}) \notag \\
 & \, -\left(\frac{1}{\mathcal{A}(B_\e)}-\frac{1}{\mathcal{A}(q(0))}\right) \left(\frac{f'(W(\frac{\delta(x)}{\e}))}{\mathcal{A}(q(0))}\Phi(\frac{\delta(x)}{\e})+f(W(\frac{\delta(x)}{\e}))\right) \notag\\
& \, - \eps(N-1) \mathcal{H}_{\Gamma_{\delta(x)}}(\sigma(x))\left(\frac{f'(W(\frac{\delta(x)}{\e}))}{\mathcal{A}(q(0))}\Psi(\frac{\delta(x)}{\e})+W'(\frac{\delta(x)}{\e})\right)\\
&\,+\underbrace{\eps(N-1)\left(\frac{1}{\mathcal{A}(B_\e)}-\frac{1}{\mathcal{A}(q(0))}\right) \mathcal{H}_{\Gamma_{\delta(x)}}(\sigma(x))\Phi'(\frac{\delta(x)}{\e})}_{:=J_{1,\eps}(x)}\notag\\
&\,+\underbrace{\eps^3(N-1)\left[\mathcal{H}_{\Gamma_{\delta(x)}}(\sigma(x))\Delta\Psi(\frac{\delta(x)}{\e})-\Delta\left(\mathcal{H}_{\Gamma_{\delta(x)}}(\sigma(x))\Psi(\frac{\delta(x)}{\e})\right)\right]}_{:=J_{2,\eps}(x)}\notag\\
=&\,\frac{1}{\mathcal{A}(B_\e)}\left(f(u(x))-f(W(\frac{\delta(x)}{\eps}))\right) -\left(\frac{1}{\mathcal{A}(B_\e)}-\frac{1}{\mathcal{A}(q(0))}\right) \frac{f'(W(\frac{\delta(x)}{\e}))}{\mathcal{A}(q(0))}\Phi(\frac{\delta(x)}{\e})\notag \\
& \, - \e (N-1) \mathcal{H}_{\Gamma_{\delta(x)}}(\sigma(x)) \frac{f'(W(\frac{\delta(x)}{\e}))}{\mathcal{A}(q(0))}\Psi(\frac{\delta(x)}{\e})  + J_{1,\eps}(x)+J_{2,\eps}(x)\notag\\
     = &\, \frac{f'(W(\frac{\delta(x)}{\e}))}{\mathcal{A}(q(0))} U_{\eps}(x) +\underbrace{\left(\frac{1}{\mathcal{A}(B_\e)}-\frac{1}{\mathcal{A}(q(0))}\right)\left(f(u(x))-f(W(\frac{\delta(x)}{\eps}))\right)}_{:=J_{3,\eps}(x)}\notag\\
     &\,+\underbrace{\frac{1}{\mathcal{A}(q(0))}\left[f(u(x))-f(W(\frac{\delta(x)}{\eps}))-f'(W(\frac{\delta(x)}{\eps}))\left(u(x)-W(\frac{\delta(x)}{\eps})\right)\right]}_{:=J_{4,\eps}(x)}+J_{1,\eps}(x)+J_{2,\eps}(x).\notag
  \end{align*}
As a consequence,
 \begin{equation}\label{30U}
     \eps^2\Delta U_{\eps}(x)= \frac{f'(W(\frac{\delta(x)}{\e}))}{\mathcal{A}(q(0))} U_{\eps}(x)+   J_{\eps}(x),~ x\in\Omega_{d_*},~\text{where}~J_{\eps}:=\sum_{i=1}^4J_{i,\eps}.
 \end{equation}
Furthermore, by \eqref{assume-f} and \eqref{int-est-u} we have 
 \begin{equation}\label{cur-e}
\eps^2\Delta{U}^2_{\eps}
 \geq \frac{2\min_{[0,b_*]}f'}{\mathcal{A}(q(0))}{U}^2_{\eps}+2U_{\eps}J_{\eps}\geq \frac{\min_{[0,b_*]}f'}{\mathcal{A}(q(0))}{U}^2_{\eps}-\frac{\mathcal{A}(q(0))}{\min_{[0,b_*]}f'}J_{\eps}^2\quad\mathrm{in}\,\,\Omega_{d_*}.
 \end{equation}

 We next estimate $U_{\eps}^2$ on $\partial\Omega$. 
Since $\partial_{\vec{n}}\delta = -1$ on $\partial\Omega$, it follows from \eqref{bd}, \eqref{ode-W}, \eqref{Phi-new} and \eqref{Psi-new} that, for each $\sigma \in \partial\Omega$,
\begin{align*}
    U_{\eps}(\sigma)+\gamma\eps\partial_{\vec{n}}U_{\eps}(\sigma)=&\,b_0-(W(0)-\gamma{W}'(0))-\left(\frac{1}{\mathcal{A}(B_\e)}-\frac{1}{\mathcal{A}(q(0))}\right)(\Phi(0)-\gamma\Phi'(0))\\
    &\,-\eps(N-1)\mathcal{H}_{\partial\Omega}(\sigma)(\Psi(0)-\gamma\Psi'(0))=0.
\end{align*}
 This implies
\begin{equation}\label{u-w-bd1}
U_{\eps}^2 + \frac{\gamma\eps}{2}\,\partial_{\vec{n}}(U_{\eps}^2) = 0
\qquad \text{on } \partial\Omega.
\end{equation}

On the other hand, when $\delta(x)=d_*$, it follows from \eqref{int-est-u}, \eqref{w3w}, \eqref{ae}, and \eqref{cw-exp} that
\begin{equation}\label{bd*}
\begin{aligned}
U_{\eps}^2(x)
\leq&\,4\Biggl(
u^2(x)
+ W^2(\frac{d_*}{\eps})
+ \left(\frac{1}{\mathcal{A}(B_\e)}-\frac{1}{\mathcal{A}(q(0))}\right)^2
\Phi^2(\frac{d_*}{\e}) \\
&\quad+ \eps^2 (N-1)^2 \bigl(\mathcal{H}_{\Gamma_{d_*}}(\sigma(x))\bigr)^2
\Psi^2(\frac{d_*}{\e})
\Biggr) \\
\lesssim&\, \exp\!\left(-M\frac{d_*}{\eps}\right)
\qquad \text{for some } M>0 \text{ independent of } \eps.
\end{aligned}
\end{equation}
Consequently, if $U_{\eps}^2$ attains its maximum on $\partial\Omega_{d_*}$, then \eqref{u-w-bd1}--\eqref{bd*} yield
\[
\max_{\overline{\Omega_{d_*}}} |U_{\eps}|
\lesssim \exp\!\left(-\frac{M d_*}{2\eps}\right)
\ll \eps^2
\qquad \text{as } 0<\eps\ll1.
\]

It remains to consider the case where $U_{\eps}^2$ attains its maximum at an interior point of $\Omega_{d_*}$. 
In this case, combining the estimate \eqref{cur-e} with the maximum principle yields
\begin{align}\label{UaJ}
\max_{\overline{\Omega_{d_*}}}|U_{\eps}|
\leq
\frac{\mathcal{A}(q(0))}{\min_{[0,b_*]}f'}
\max_{\overline{\Omega_{d_*}}}|J_{\eps}|.
\end{align}
Note that both $\Delta\mathcal{H}_{\Gamma_{\delta(x)}}(\sigma(x))$ and $|\nabla\mathcal{H}_{\Gamma_{\delta(x)}}(\sigma(x))|$ are uniformly bounded in $\overline{\Omega_{d_*}}$, since $\Omega$ is bounded and $\partial\Omega$ is sufficiently smooth.
We now establish the following estimates to derive an upper bound for $\max_{\overline{\Omega_{d_*}}}|J_{\eps}|$:
  \begin{align*}
      \max_{x\in\overline{\Omega_{d_*}}}|J_{1,\eps}(x)|\leq&\,\frac{2\eps}{N-1}\left|\frac{1}{\mathcal{A}(B_\e)}-\frac{1}{\mathcal{A}(q(0))}\right|\left(\sum_{i=1}^{N-1}\max_{\sigma\in\partial\Omega}|\kappa_i(\sigma)|\right)\sup_{[0,\infty)}|\Phi'|\\
      \lesssim&\,\eps^2\quad(\mathrm{by\,\,\eqref{d-star},\,\,\eqref{me-H},\,\,\eqref{ae}\,\,and\,\,Lemma~\ref{lem-PhiPsi}(i)});\\[1em]
      \max_{x\in\overline{\Omega_{d_*}}}|J_{2,\eps}(x)|=&\,\eps^3(N-1)\max_{x\in\overline{\Omega_{d_*}}}\left|\Psi(\frac{\delta(x)}{\e})\Delta\mathcal{H}_{\Gamma_{\delta(x)}}(\sigma(x))+\frac{2}{\eps}\Psi'(\frac{\delta(x)}{\e})\nabla\mathcal{H}_{\Gamma_{\delta(x)}}(\sigma(x))\cdot\nabla\delta(x)\right|\\
      \lesssim&\,\eps^2\quad\mathrm{(by\,\,Lemma~\ref{lem-PhiPsi}(ii))};\\[1em]
      \max_{x\in\overline{\Omega_{d_*}}}|J_{3,\eps}(x)|\leq&\,\left|\frac{1}{\mathcal{A}(B_\e)}-\frac{1}{\mathcal{A}(q(0))}\right|\left(\max_{[0,b_*]}f'\right)\max_{x\in\overline{\Omega_{d_*}}}\left|u(x)-W(\frac{\delta(x)}{\eps})\right|\\
      \lesssim&\,\eps^2\quad\mathrm{(by\,\,\eqref{assume-f},\,\,\eqref{u-to-W},\,\,\eqref{w1w}\,\,and\,\,\eqref{ae})};\\[1em]
       \max_{x\in\overline{\Omega_{d_*}}}|J_{4,\eps}(x)|\leq&\,\frac{1}{\mathcal{A}(q(0))}\left(\max_{[0,b_*]}|f''|\right)\max_{x\in\overline{\Omega_{d_*}}}\left(u(x)-W(\frac{\delta(x)}{\eps})\right)^2\lesssim\eps^2\quad\mathrm{(by\,\,\eqref{u-to-W}\,\,and\,\,\eqref{w1w})}.
  \end{align*}
 Here we have used the estimate $\delta(x)\leq d_*\leq\min_{1 \le i \le N-1}1/{\left(1 + 2 
 \sup_{\sigma \in \partial\Omega}|\kappa_i(\sigma)|\right)}$ in $\overline{\Omega_{d_*}}$ (by \eqref{d-star}).  With these bounds at hand, \eqref{UaJ} yields
\begin{align}\label{JJ}
\max_{\overline{\Omega_{d_*}}}|U_{\eps}|
\leq
\frac{\mathcal{A}(q(0))}{\min_{[0,b_*]}f'}
\max_{\overline{\Omega_{d_*}}}|J_{\eps}|
\leq\frac{\mathcal{A}(q(0))}{\min_{[0,b_*]}f'}
\max_{\overline{\Omega_{d_*}}} \sum_{i=1}^4 |J_{i,\eps}|
\lesssim \eps^2.
\end{align}
Combining \eqref{JJ} with the refined asymptotic estimate $\frac{1}{\mathcal{A}(B_\e)}-\frac{1}{\mathcal{A}(q(0))}
=\frac{\mathcal{A}(q(0))-\mathcal{A}(B_\e)}{\left(\mathcal{A}(q(0))\right)^2}
+O(\eps^2)$ as $0<\eps\ll1$, which follows directly from \eqref{ae}, we conclude that \eqref{moon-1}--\eqref{Xi0} holds.

Let $x_{\varepsilon}\in\overline{\Omega_{d_*}}$ depend on $\varepsilon$ and satisfy 
$\delta(x_{\varepsilon})=O(\varepsilon)$ as $0<\varepsilon\ll1$. 
Then, by \eqref{d-star} and \eqref{me-H}, we have $\big|\mathcal{H}_{\partial\Omega}(\sigma(x_{\varepsilon}))
-\mathcal{H}_{\Gamma_{\delta(x_{\varepsilon})}}(\sigma(x_{\varepsilon}))\big|
\lesssim \delta(x_{\varepsilon})
\lesssim \varepsilon.$ Combining this estimate with \eqref{moon-1}, we obtain \eqref{udstar}. 
This completes the proof of Lemma~\ref{lem2}.
\end{proof}

\subsection{ Completion of the proof of Theorem~\ref{thm-nonlocal}}

\eqref{moon-1} indicates that, for sufficiently small $\varepsilon>0$, 
the nonlocal contribution arising from $\mathcal{A}(B_\varepsilon)-\mathcal{A}(q(0))=O(\varepsilon)$ 
plays a significant role in the geometric aspects of the asymptotic behavior of $u$. 
Therefore, a central analytical task is to determine the precise leading-order behavior of 
$\varepsilon^{-1}\big(\mathcal{A}(B_\varepsilon)-\mathcal{A}(q(0))\big)$ as $\varepsilon \downarrow 0$. 
To this end, we recall the following asymptotic expansion based on the coarea formula.
\begin{lemma}[cf.~Lemma~2.1 of \cite{LMWY2025}]\label{coarea-lem}
Let $h:\overline{\Omega}\to\mathbb{R}$ be a smooth function. Then for each $d\in(0,d_*]$, 
the domain $\Omega_d\subset\Omega_{d_*}$ satisfies property (G-i). Moreover, for each $x\in\Omega_d$, 
$h(x)$ can be expressed as $h(\sigma(x)-\delta(x)\vec{n}(\sigma(x)))=\widetilde{h}(\delta(x), \sigma(x))$ in the sense of \eqref{x-to=sigma}. 
The following expansion holds:
\begin{equation}\label{coarea}
\int_{\Omega_d} h(x)\,\mathrm{d}x 
=
\int_0^d \int_{\partial\Omega} 
\widetilde{h}(\delta, \sigma) 
\Big(1-(N-1)\mathcal{H}_{\partial\Omega}(\sigma)\,\delta\Big) 
\, \mathrm{d}\sigma \mathrm{d}\delta+O(d^3),
\end{equation}
where $O(d^3)$ denotes a quantity bounded by $C d^3$ for some constant $C>0$ independent of $d\in[0,d_*]$.
\end{lemma}

 \eqref{coarea} is a generalization of Weyl's tube formula~\cite{W1939}:
 \begin{equation}\label{weyl}
  |\Omega_d| = d|\partial\Omega| - \frac{d^2}{2}(N-1)\int_{\partial\Omega}\mathcal{H}_{\partial\Omega}(\sigma) \text{d}\sigma + O(d^3),~\text{as}~0<d\ll 1.   
 \end{equation}
A proof of \eqref{coarea} can be found in \cite[Lemma~2.1]{LMWY2025}, and is therefore omitted here.

To complete the proof of Theorem~\ref{thm-nonlocal}, we proceed in two steps.\\

\noindent
{\bf Step~1 - Refining the asymptotic expansion of the nonlocal term.}

 In this step, we derive a refined asymptotic expansion of 
$B_\varepsilon=\fint_{\Omega}q(u)\,\mathrm{d}x$ as $\varepsilon\to0^+$ 
by combining \eqref{moon-1}--\eqref{Xi0} with the coarea expansion \eqref{coarea}. 
We begin by establishing estimates for $q(u)$ in $\Omega$. 
Recall that $u(x)\in(0,b_0)$ and $W(t)\in(0,b_*]$. 
For any fixed $d\in(0,d_*)$, it follows from \eqref{u-to-W}, \eqref{ae}, and Lemma~\ref{lem2} that

\begin{align}\label{*qu*}
\max_{x\in\overline{\Omega_{d}}}\left|q(u(x))-q(W(\frac{\delta(x)}{\e}))-\eps  q'(W(\frac{\delta(x)}{\e}))\boldsymbol{\Xi}_\e(\delta(x),\sigma(x))\right|\lesssim\e^2, 
\end{align}
where $\boldsymbol{\Xi}_\e(\delta(x),\sigma(x))$ describing the geometric effect is defined by \eqref{Xi0}. Note that \eqref{int-est-u} yields
\begin{align}\label{qu0}
\sup_{x\in\overline{\Omega}\setminus\overline{\Omega_{d}}}\left|q(u(x))-q(0)\right|
\lesssim \mathrm{e}^{-\frac{C_*}{\eps}d}.
\end{align}
Since $B_{\varepsilon}=\fint_{\Omega}q(u)\,\mathrm{d}x$, 
 \eqref{*qu*} and \eqref{qu0} indicate that a refined asymptotic expansion of 
$\int_{\Omega}q(u)\,\mathrm{d}x$ can be obtained by decomposing the domain into the boundary layer region and the interior. 
To this end, we choose $d_{\varepsilon}\in(0,d_*)$ depending on $\varepsilon$ such that
\begin{equation}\label{d-infi}
\lim_{\varepsilon\downarrow0}\varepsilon^{-1}d_\varepsilon=\infty
\quad\text{and}\quad
\lim_{\varepsilon\downarrow0}\varepsilon^{-\frac{2}{3}}d_\varepsilon=0,
\end{equation}
where the latter condition is crucial for the subsequent estimates. Under this choice, the contribution from $\Omega\setminus\overline{\Omega_{d_{\eps}}}$ is exponentially small, namely~$\int_{\Omega\setminus\overline{\Omega_{d_{\eps}}}}(q(u)-q(0))\,\mathrm{d}x
=O\left(\mathrm{e}^{-\frac{C_*}{\eps}d_{\eps}}\right).$
Therefore, using \eqref{b-eps} together with the expansion \eqref{*qu*} in $\Omega_{d_{\eps}}$, we obtain
\begin{equation}\label{qB*}
\begin{aligned}
\left(B_\e-q(0)\right)|\Omega|
=&\,\left\{\int_{\Omega_{d_{\eps}}}+\int_{\Omega\setminus\overline{\Omega_{d_{\eps}}}}\right\}(q(u)-q(0))\,\mathrm{d}x \\
=&\,\int_{\Omega_{d_{\eps}}}(q(u)-q(0))\,\mathrm{d}x
+O\left(\mathrm{e}^{-\frac{C_*}{\eps}d_{\eps}}\right)\\
=&\,\int_{\Omega_{d_{\eps}}} \left(q(W(\frac{\delta(x)}{\e}))-q(0)\right)\mathrm{d}x\\
&+\frac{\mathcal{A}(q(0))-\mathcal{A}(B_\e)}{(\mathcal{A}(q(0)))^2}
\int_{\Omega_{d_{\eps}}} q'(W(\frac{\delta(x)}{\e}))
\Phi(\frac{\delta(x)}{\e})\mathrm{d}x \\
&+\e (N-1)
\int_{\Omega_{d_{\eps}}}
q'(W(\frac{\delta(x)}{\e}))
\Psi(\frac{\delta(x)}{\e})
\mathcal{H}_{\Gamma_{\delta(x)}}(\sigma(x))\,\mathrm{d}x + O(\e^2)|\Omega_{d_{\eps}}|.
\end{aligned}
\end{equation}
We proceed by applying Lemma~\ref{coarea-lem} to each term in \eqref{qB*} to derive a refined asymptotic expansion of $B_{\eps}$ as $\eps\to0^+$. Firstly, one may check that

\begin{equation}\label{qw-1st}
    \begin{aligned}
    \int_{\Omega_{d_{\eps}}}&\, \left(q(W(\frac{\delta(x)}{\e}))-q(0)\right)\mathrm{d}x\\
    =&\,\int_0^{d_{\eps}} \int_{\partial\Omega} \left(q(W(\frac{\delta}{\e}))-q(0)\right) \Big(1-(N-1)\mathcal{H}_{\partial\Omega}(\sigma)\delta \Big)\, \mathrm{d}\sigma \mathrm{d}\delta+O(d_\eps^3)\\  =&\,|\partial\Omega|\int_0^{d_{\eps}}  \left(q(W(\frac{\delta}{\e}))-q(0)\right)\,  \mathrm{d}\delta\\
    &\,\,-(N-1)\int_{\partial\Omega}\mathcal{H}_{\partial\Omega}(\sigma)\, \mathrm{d}\sigma\int_0^{d_{\eps}}\delta\left(q(W(\frac{\delta}{\e}))-q(0)\right)\mathrm{d}\delta+O(d_\eps^3)\\ =&\,\eps|\partial\Omega|\int_0^{\frac{d_{\eps}}{\eps}}  \left(q(W(t)-q(0)\right)\,  \mathrm{d}t\\
    &-\eps^2(N-1)\int_{\partial\Omega}\mathcal{H}_{\partial\Omega}(\sigma)\, \mathrm{d}\sigma\int_0^{\frac{d_{\eps}}{\eps}}t\left(q(W(t))-q(0)\right)\mathrm{d}t+O(d_\eps^3),
    \end{aligned}
\end{equation}
and
\begin{equation}\label{qw-2nd}
\begin{aligned}
\int_{\Omega_{d_{\eps}}} q'(W(\frac{\delta(x)}{\e}))\Phi(\frac{\delta(x)}{\e})\,\mathrm{d}x
    =&\,\eps|\partial\Omega|\int_0^{\frac{d_{\eps}}{\eps}} q'(W(t))\Phi(t)\,  \mathrm{d}t\\
    &\,-\eps^2(N-1)\int_{\partial\Omega}\mathcal{H}_{\partial\Omega}(\sigma)\, \mathrm{d}\sigma\int_0^{\frac{d_{\eps}}{\eps}}tq'(W(t))\Phi(t)\,\mathrm{d}t+O(d_\eps^3).
\end{aligned}
\end{equation}
Note also that $\max\limits_{x\in\Omega_{d_{\eps}}}|\mathcal{H}_{\Gamma_{\delta(x)}}(\sigma(x))-\mathcal{H}_{\partial\Omega}(\sigma(x))|=O(d_{\eps})$ (by \eqref{d-star} and \eqref{me-H}). Thus, we have
\begin{equation}\label{qw-3rd}
\begin{aligned}
 \int_{\Omega_{d_{\eps}}}&\,  q'(W(\frac{\delta(x)}{\e}))\Psi(\frac{\delta(x)}{\e}) \mathcal{H}_{\Gamma_{\delta(x)}}(\sigma(x))\,\mathrm{d}x\\
 =&\,\int_{\Omega_{d_{\eps}}}  q'(W(\frac{\delta(x)}{\e}))\Psi(\frac{\delta(x)}{\e}) \left(\mathcal{H}_{\partial\Omega}(\sigma(x))+O(d_{\eps})\right)\mathrm{d}x\\
=&\,\eps\left(\int_{\partial\Omega}\mathcal{H}_{\partial\Omega}(\sigma)\mathrm{d}\sigma+O(d_{\eps})\right) \left( \int_0^{\frac{d_{\eps}}{\eps}}q'(W(t))\Psi(t)\,  \mathrm{d}t+O(d_\eps^3)\right)\\
&-\eps^2(N-1)\left(\int_{\partial\Omega}\mathcal{H}^2_{\partial\Omega}(\sigma)\, \mathrm{d}\sigma+O(d_{\eps})\right)\int_0^{\frac{d_{\eps}}{\eps}} tq'(W(t))\Psi(t)\, \mathrm{d}t.
\end{aligned}
\end{equation}

To derive the refined asymptotic expansion of $B_{\varepsilon}$ from \eqref{qB*}--\eqref{qw-3rd}, we first establish the following preliminary estimates.
\begin{lemma}\label{qw-lem}
Let $d_{\eps}$ satisfy \eqref{d-infi}. Then 
    \begin{align}
    \int_0^{\frac{d_{\eps}}{\eps}}  \left(q(W(t)-q(0)\right)\,  \mathrm{d}t=&\,\sqrt{\mathcal{A}(q(0))} \mathcal{Q}_F(b_*) +O(\mathrm{e}^{-M_0\frac{d_{\eps}}{\eps}}),\label{0qW}\\
    \int_0^{\frac{d_{\eps}}{\eps}}  t\left(q(W(t)-q(0)\right)\,  \mathrm{d}t=&\,\sqrt{\mathcal{A}(q(0))}\int_0^\infty \mathcal{Q}_F(W(t))\,\mathrm{d}t+O(\frac{d_{\eps}}{\eps}\mathrm{e}^{-M_0\frac{d_{\eps}}{\eps}}),\label{1qW}\\
    \int_0^{\frac{d_{\eps}}{\eps}}  t^2\left(q(W(t)-q(0)\right)\,  \mathrm{d}t=&\,2{\mathcal{A}(q(0))}\int_0^\infty \widetilde{\mathcal{Q}}_F(W(t))\,\mathrm{d}t+O(\frac{d_{\eps}^2}{\eps^2}\mathrm{e}^{-M_0\frac{d_{\eps}}{\eps}}),\label{2qW}\\
   \int_0^{\frac{d_{\eps}}{\eps}}t^kq'(W(t))\Phi(t)\,\mathrm{d}t    =&\,\int_0^{\infty}t^kq'(W(t))\Phi(t)\,\mathrm{d}t+O(\frac{d_{\eps}^k}{\eps^k}\mathrm{e}^{-\widetilde{M}\frac{d_{\eps}}{\eps}}),\label{kqh}\\
    \int_0^{\frac{d_{\eps}}{\eps}}t^kq'(W(t))\Psi(t)\,\mathrm{d}t
    =&\,\int_0^{\infty}t^kq'(W(t))\Psi(t)\,\mathrm{d}t+O(\frac{d_{\eps}^k}{\eps^k}\mathrm{e}^{-\widetilde{M}\frac{d_{\eps}}{\eps}}),\quad\,k=0,1,2,\label{kqs}
\end{align}
where $M_0$ is a positive constant, independent of $\varepsilon$, defined in Lemma~\ref{lem-W}; $\mathcal{Q}_F$ is given by \eqref{Q-F}, $\widetilde{\mathcal{Q}}_F(t):=\int_0^t\frac{\mathcal{Q}_F(s)}{\sqrt{2F(s)}}\,\mathrm{d}s$
and $\widetilde{M}$ is defined in \eqref{cw-exp}.
\end{lemma}
The proof of Lemma~\ref{qw-lem} is given in Appendix~A (see Section~\ref{subap3}).

Note that \eqref{d-infi} implies $d_{\eps}\xrightarrow{\eps\downarrow0}0$ and $\frac{d_{\eps}}{\eps}\xrightarrow{\eps\downarrow0}\infty$. Thus, 
by  \eqref{d-infi}, \eqref{qw-1st}--\eqref{qw-3rd} and Lemma~\ref{qw-lem}, we have the following asymptotic expansions for each term in \eqref{qB*}:
\begin{equation}\label{stp-1}
  \begin{aligned}
 \int_{\Omega_{d_{\eps}}}& \left(q(W(\frac{\delta(x)}{\e}))-q(0)\right)\mathrm{d}x\\
 =&\,\eps|\partial\Omega|\left(\sqrt{\mathcal{A}(q(0))} \mathcal{Q}_F(b_*) +O(\eps^2)\right) \\ &\,-\,\eps^2(N-1)\int_{\partial\Omega}\mathcal{H}_{\partial\Omega}(\sigma)\, \mathrm{d}\sigma\left(\sqrt{\mathcal{A}(q(0))}\int_0^\infty \mathcal{Q}_F(W(t))\,\mathrm{d}t+O(\frac{d_{\eps}}{\eps}\mathrm{e}^{-M_0\frac{d_{\eps}}{\eps}})\right)\\
 =&\,\sqrt{\mathcal{A}(q(0))}\left(\eps|\partial\Omega| \mathcal{Q}_F(b_*)\,-\,\eps^2(N-1)\int_{\partial\Omega}\mathcal{H}_{\partial\Omega}(\sigma)\, \mathrm{d}\sigma\int_0^\infty \mathcal{Q}_F(W(t))\,\mathrm{d}t+O(\eps^3)\right),
\end{aligned}  
\end{equation}
\begin{equation}\label{stp-2}
    \begin{aligned}
    \int_{\Omega_{d_{\eps}}} q'(W(\frac{\delta(x)}{\e}))\Phi(\frac{\delta(x)}{\e})\,\mathrm{d}x=&\,\eps|\partial\Omega|\int_0^{\infty}q'(W(t))\Phi(t)\,\mathrm{d}t\\
    &-\,\eps^2(N-1)\int_{\partial\Omega}\mathcal{H}_{\partial\Omega}(\sigma)\, \mathrm{d}\sigma\int_0^{\infty}tq'(W(t))\Phi(t)\,\mathrm{d}t+O(\eps^3),
    \end{aligned}
\end{equation}
and
\begin{equation}\label{stp-3}
\begin{aligned}
 \int_{\Omega_{d_{\eps}}} &\, q'(W(\frac{\delta(x)}{\e}))\Psi(\frac{\delta(x)}{\e}) \mathcal{H}_{\Gamma_{\delta(x)}}(\sigma(x))\,\mathrm{d}x\\
=&\,\eps\left(\int_{\partial\Omega}\mathcal{H}_{\partial\Omega}(\sigma)\mathrm{d}\sigma+O(d_{\eps})\right) \left( \int_0^{\infty}q'(W(t))\Psi(t)\,  \mathrm{d}t  +O(\eps^2)\right)\\
&-\eps^2(N-1)\left(\int_{\partial\Omega}\mathcal{H}^2_{\partial\Omega}(\sigma)\, \mathrm{d}\sigma+O(d_{\eps})\right)\left(\int_0^{\infty}tq'(W(t))\Psi(t)\,\mathrm{d}t+O(\frac{d_{\eps}}{\eps}\mathrm{e}^{-\widetilde{M}\frac{d_{\eps}}{\eps}})\right).
\end{aligned}
\end{equation}

Note that, by \eqref{weyl} and \eqref{d-infi}, $|\Omega_{d_\varepsilon}|\ll\varepsilon^{\frac{2}{3}}$  as $\varepsilon \downarrow 0$. 
Combining \eqref{qB*} with \eqref{stp-1}--\eqref{stp-3}, we derive the asymptotic expansion of each term up to order $\varepsilon^2$, with the remaining contributions of order $\varepsilon^{\frac83}$ as follows:
\begin{equation}\label{*qB*}
{\begin{aligned}
&\left(B_\e-q(0)\right)|\Omega|\\
=&\,\int_{\Omega_{d_{\eps}}} \left(q(W(\frac{\delta(x)}{\e}))-q(0)\right)\mathrm{d}x +\frac{\mathcal{A}(q(0))-\mathcal{A}(B_\e)}{\left(\mathcal{A}(q(0))\right)^2}\int_{\Omega_{d_{\eps}}} q'(W(\frac{\delta(x)}{\e}))\Phi(\frac{\delta(x)}{\e})\,\mathrm{d}x \\
&+\eps (N-1) \int_{\Omega_{d_{\eps}}} q'(W(\frac{\delta(x)}{\e}))\Psi(\frac{\delta(x)}{\e}) \mathcal{H}_{\Gamma_{\delta(x)}}(\sigma(x))\,\mathrm{d}x + O(\e^2)|\Omega_{d_{\eps}}|\\
=&\,\eps\sqrt{\mathcal{A}(q(0))}\left(|\partial\Omega| \mathcal{Q}_F(b_*)\,-\,\eps(N-1)\int_{\partial\Omega}\mathcal{H}_{\partial\Omega}(\sigma)\, \mathrm{d}\sigma\int_0^\infty \mathcal{Q}_F(W(t))\,\mathrm{d}t+O(\eps^2)\right)\\
&+\eps\left(\frac{\mathcal{A}(q(0))-\mathcal{A}(B_\e)}{\left(\mathcal{A}(q(0))\right)^2}\right)\text{\scriptsize$\displaystyle\left[|\partial\Omega|\int_0^{\infty}q'(W(t))\Phi(t)\,\mathrm{d}t
    \,-\,\eps(N-1)\int_{\partial\Omega}\mathcal{H}_{\partial\Omega}(\sigma)\, \mathrm{d}\sigma\int_0^{\infty}tq'(W(t))\Phi(t)\,\mathrm{d}t+O(\eps^2)\right]$}\\
&+\eps^2(N-1)\left[\left(\int_{\partial\Omega}\mathcal{H}_{\partial\Omega}(\sigma)\mathrm{d}\sigma+O(d_{\eps})\right)  \int_0^{\infty}q'(W(t))\Psi(t)\,  \mathrm{d}t  +O(\eps)\right] +  O(\e^{\frac83}).
\end{aligned}}\end{equation}
\noindent
{\bf Step~2. Second- and third-order asymptotics of $B_{\varepsilon}$.}

We analyze each term in \eqref{*qB*}, recalling that $B_{\varepsilon}\to q(0)$ and $d_{\varepsilon}\to0$ as $\varepsilon\downarrow0$. 
Since $\mathcal{A}(B_\varepsilon)$ also appears in \eqref{*qB*}, it is natural to seek an expansion of $B_{\varepsilon}$ of the form
\begin{align}\label{b(i)}
B_{\varepsilon}
=
q(0)
+\boldsymbol{B}_{\boldsymbol{1}}\varepsilon
+\left(\boldsymbol{B}_{\boldsymbol{2}}+\boldsymbol{\mathrm{o}}_{\varepsilon}\right)\varepsilon^2,
\qquad 0<\varepsilon\ll1,
\end{align}
where $\boldsymbol{B}_{\boldsymbol{1}}$ and $\boldsymbol{B}_{\boldsymbol{2}}$ are constants independent of $\varepsilon$, and $\boldsymbol{\mathrm{o}}_{\varepsilon}\to0$ as $\varepsilon\downarrow0$ (see (N-iii)). 
We determine the coefficients $\boldsymbol{B}_{\boldsymbol{1}}$ and $\boldsymbol{B}_{\boldsymbol{2}}$ by matching the expansions in \eqref{*qB*}.

Since $\mathcal{A}\in\mathrm{C}^3(\mathbb{R};\mathbb{R}^+)$ is positive (cf. \eqref{q-A}), by \eqref{b(i)} one has
\begin{equation}\label{qAm}
  \begin{aligned}
 \mathcal{A}(B_{\eps})=&\,\mathcal{A}(q(0))+\mathcal{A}'(q(0))(B_{\eps}-q(0))+\frac{\mathcal{A''}(q(0))}{2}(B_{\eps}-q(0))^2+O((B_{\eps}-q(0))^3)   \\
 =&\,\mathcal{A}(q(0))\left[1+\frac{\mathcal{A}'(q(0))}{\mathcal{A}(q(0))}\boldsymbol{B}_{\boldsymbol{1}}\eps+\left(\frac{\mathcal{A}'(q(0))}{\mathcal{A}(q(0))}\boldsymbol{B}_{\boldsymbol{2}}+\frac{\mathcal{A}''(q(0))}{2\mathcal{A}(q(0))}\boldsymbol{B}_{\boldsymbol{1}}^2\right)\eps^2+O(\eps^3)\right],
 \end{aligned}  
\end{equation}
and
\begin{align}\label{mAq}
\frac{\mathcal{A}(q(0))-\mathcal{A}(B_\e)}{(\mathcal{A}(q(0)))^2}=-\frac{\mathcal{A}'(q(0))}{(\mathcal{A}(q(0)))^2}\boldsymbol{B}_{\boldsymbol{1}}\eps+O(\eps^2),
\end{align}
as $0<\eps\ll1$. Putting \eqref{b(i)} and \eqref{mAq} into \eqref{*qB*} and using \eqref{d-infi}, we arrive at
\begin{equation}\label{cqB}
\begin{aligned}
&\left(\boldsymbol{B}_{\boldsymbol{1}}\eps+(\boldsymbol{B}_{\boldsymbol{2}}+\boldsymbol{\mathrm{o}}_\eps)\eps^2\right)|\Omega|\\
&\quad=\eps\sqrt{\mathcal{A}(q(0))}\left(|\partial\Omega| \mathcal{Q}_F(b_*)\,-\,\eps(N-1)\int_{\partial\Omega}\mathcal{H}_{\partial\Omega}(\sigma)\, \mathrm{d}\sigma\int_0^\infty \mathcal{Q}_F(W(t))\,\mathrm{d}t\right)\\
&\quad\quad-\eps^2\frac{\mathcal{A}'(q(0))}{(\mathcal{A}(q(0)))^2}\boldsymbol{B}_{\boldsymbol{1}}|\partial\Omega|\int_0^{\infty}q'(W(t))\Phi(t)\,\mathrm{d}t\\
&\quad\quad+\eps^2(N-1)\int_{\partial\Omega}\mathcal{H}_{\partial\Omega}(\sigma)\mathrm{d}\sigma  \int_0^{\infty}q'(W(t))\Psi(t)\,  \mathrm{d}t+O(\eps^2d_{\eps})\\
    &\quad=\eps\sqrt{\mathcal{A}(q(0))}|\partial\Omega| \mathcal{Q}_F(b_*)+\eps^2\left[-\,(N-1)\sqrt{\mathcal{A}(q(0))}\int_{\partial\Omega}\mathcal{H}_{\partial\Omega}(\sigma)\, \mathrm{d}\sigma\int_0^\infty \mathcal{Q}_F(W(t))\,\mathrm{d}t\right.\\
    &\quad\quad-\frac{\mathcal{A}'(q(0))}{(\mathcal{A}(q(0)))^2}\boldsymbol{B}_{\boldsymbol{1}}|\partial\Omega|\int_0^{\infty}q'(W(t))\Phi(t)\,\mathrm{d}t\\
   &\quad\quad \left.+(N-1)\int_{\partial\Omega}\mathcal{H}_{\partial\Omega}(\sigma)\mathrm{d}\sigma  \int_0^{\infty}q'(W(t))\Psi(t)\,  \mathrm{d}t+O(d_{\eps})\right]\\
 &\quad=  \eps\sqrt{\mathcal{A}(q(0))}|\partial\Omega| \mathcal{Q}_F(b_*)\\
  &\quad\quad-\eps^2\left(\frac{\mathcal{A}'(q(0))}{(\mathcal{A}(q(0)))^2}\boldsymbol{B}_{\boldsymbol{1}}|\partial\Omega|\mathcal{I}_{W,\Phi}+(N-1)\int_{\partial\Omega}\mathcal{H}_{\partial\Omega}(\sigma)\mathrm{d}\sigma\mathcal{J}_{W,\Psi}+O(d_{\eps})\right), 
\end{aligned}\end{equation}
where $\mathcal{I}_{W,\Phi}$ and $\mathcal{J}_{W,\Psi}$ were defined by \eqref{I-J}. Consequently, \eqref{cqB} gives
\begin{align}\label{330-b1}
\boldsymbol{B}_{\boldsymbol{1}}=\sqrt{\mathcal{A}(q(0))} \mathcal{Q}_F(b_*)\frac{|\partial\Omega|}{|\Omega|},
\end{align}
and
\begin{align}\label{330-b2}
\boldsymbol{B}_{\boldsymbol{2}}=&\,-\frac{\mathcal{A}'(q(0))}{(\mathcal{A}(q(0)))^2}\boldsymbol{B}_{\boldsymbol{1}}\mathcal{I}_{W,\Phi}\frac{|\partial\Omega|}{|\Omega|}-(N-1)\mathcal{J}_{W,\Psi}\frac{\int_{\partial\Omega}\mathcal{H}_{\partial\Omega}(\sigma)\mathrm{d}\sigma}{|\Omega|}\notag\\
    =&\,-\frac{\mathcal{A}'(q(0))}{(\mathcal{A}(q(0)))^{\frac32}} \mathcal{Q}_F(b_*)\mathcal{I}_{W,\Phi}\frac{|\partial\Omega|^2}{|\Omega|^2}-(N-1)\mathcal{J}_{W,\Psi}\frac{\int_{\partial\Omega}\mathcal{H}_{\partial\Omega}(\sigma)\mathrm{d}\sigma}{|\Omega|}.
\end{align}
Along with \eqref{b(i)}, we obtain
\begin{equation}\label{bi*}
 \begin{aligned}
    B_{\eps}=&\,q(0)+\eps\sqrt{\mathcal{A}(q(0))} \mathcal{Q}_F(b_*)\frac{|\partial\Omega|}{|\Omega|}\\
    &\,-\eps^2\left(\frac{\mathcal{A}'(q(0))}{(\mathcal{A}(q(0)))^{\frac32}} \mathcal{Q}_F(b_*)\boldsymbol{\mathcal{I}}_{W,\Phi}\frac{|\partial\Omega|^2}{|\Omega|^2}+(N-1)\boldsymbol{\mathcal{J}}_{W,\Psi}\frac{\int_{\partial\Omega}\mathcal{H}_{\partial\Omega}(\sigma)\mathrm{d}\sigma}{|\Omega|}+\boldsymbol{\mathrm{o}}_\eps\right),
\end{aligned}   
\end{equation}
as $0<\eps\ll1$. Therefore, \eqref{q-three} follows from \eqref{b-eps} and \eqref{bi*}, which completes the proof of Theorem~\ref{thm-nonlocal}.

\section{\bf Proof of Theorem~\ref{mm-thm} and Corollary~\ref{cor-lm}}\label{sec-thm1.3}
We now present the proof of Theorem~\ref{mm-thm}. 
By \eqref{qAm} and \eqref{330-b1}--\eqref{330-b2}, we obtain
\begin{equation}\label{bigaq}
\begin{aligned}
\mathcal{A}\left(\fint_\Omega q(u)\,\mathrm{d}x \right)
= \mathcal{A}(q(0))
+\varepsilon\frac{|\partial\Omega|}{|\Omega|}  
\mathcal{A}'(q(0))\sqrt{\mathcal{A}(q(0))}\mathcal{Q}_F(b_*)
+O(\varepsilon^2),
\quad \text{as } \varepsilon\downarrow0.
\end{aligned}
\end{equation}
Substituting \eqref{bigaq} into \eqref{moon-1}--\eqref{Xi0} yields \eqref{thm-u(i)}.

It remains to prove~\eqref{thm-u(ii)}.
Recall that $U_\varepsilon$ is defined by \eqref{new-u}, namely,
\[ U_{\eps}(x)= u(x)-W(\frac{\delta(x)}{\eps})-\frac{\mathcal{A}(q(0))-\mathcal{A}(B_\e)}{\left(\mathcal{A}(q(0))\right)^2}\Phi(\frac{\delta(x)}{\e})- \e (N-1) \mathcal{H}_{\Gamma_{\delta(x)}}(\sigma(x))\Psi(\frac{\delta(x)}{\e})\]
and that $U_\eps$ satisfies the equation~\eqref{30U} and the boundary condition~\eqref{u-w-bd1}.
Moreover, $|U_\eps(x)|, |J_\eps (x)| \lesssim \eps^2$ in $\Omega_{d_*}$. (See the proof of Lemma \ref{lem2}.) Accordingly, we introduce the rescaled function
\[
\widetilde{U}_\varepsilon(y):=U_\varepsilon(\varepsilon y),
\qquad y\in \eps^{-1}\Omega.
\]
Then, by \eqref{30U} and~\eqref{u-w-bd1}, $\widetilde{U}_\varepsilon$ satisfies
\[
\Delta_y \widetilde{U}_{\varepsilon}(y)
= \frac{f'(W(\frac{\delta(\varepsilon y)}{\varepsilon}))}{\mathcal{A}(q(0))}
\widetilde{U}_{\varepsilon}(y)
+ J_{\varepsilon}(\varepsilon y),~\text{in}~\eps^{-1}\Omega_{d_*},
\quad 
\widetilde{U}_\varepsilon+\gamma \partial_{\vec{n}}\widetilde{U}_\varepsilon = 0
\quad \text{on } \partial(\eps^{-1}\Omega).
\]
We first consider the case $\gamma=0$.
For any $x_0 \in \partial\Omega$, applying the $\text{L}^p$-estimate with $p \in (1,\infty)$
(see \cite[Theorem~9.11]{GT2001}), we obtain 
\[\|\widetilde{U}_{\eps}\|_{\text{W}^{2,p}(B(y_0,R) \cap (\eps^{-1}\Omega_{d_*}))} \le C (\|\widetilde{U}_\eps\|_{\text{L}^p(B(y_0,2R) \cap (\eps^{-1}\Omega_{d_*}))} 
+ \|J_\eps(\eps \cdot)\|_{\text{L}^p(B(y_0,2R) \cap (\Omega_{d_*}/\eps))}) \lesssim \eps^2\]
for some constant $C$, independent of $x_0 \in \partial\Omega$ and small $\eps$, where $y_0=\frac{x_0}{\eps}$. 
By the Sobolev embedding, we have $|\nabla_y \widetilde{U}_\eps (y)| = O(\eps^2)$ for $y \in \eps^{-1}(\Omega_{d^*}\setminus \Omega_{\eps d^*})$.
Thus, it holds that $|\nabla U_\eps(x)|=O(\eps)$ for $x \in \Omega_{d^*}\setminus \Omega_{\eps d^*}$.

Now we assume $\gamma>0$. First, we recall that $u=u_\eps \in \text{W}^{2,p}(\Omega)$ for any $p \in (1,\infty)$ by \cite[Theorem~6.30]{L2013}. 
Moreover, from the $\text{L}^p$-estimates (see \cite[Theorem 6.27]{L2013}), we get an estimate
\[\|D^2\widetilde{U}_{\eps}\|_{L^{p}(B(y_0,R) \cap (\Omega_{d_*}/\eps))} \le C (\|\widetilde{U}_\eps\|_{\text{L}^p(B(y_0,2R) \cap (\Omega_{d_*}/\eps))} 
+ \|J_\eps(\eps \cdot)\|_{\text{L}^p(B(y_0,2R) \cap (\Omega_{d_*}/\eps))}) \lesssim \eps^2\]
for some constant $C$, independent of $x_0 \in \partial\Omega$ and small $\eps>0$.  Thus, $|\nabla U_\eps (x)|=O(\eps)$ for $x \in \Omega_{\eps d_*}$  
and hence \eqref{thm-u(ii)} also holds for $\gamma>0$.
This completes the proof of Theorem~\ref{mm-thm}.

\section{\bf Appendix: Proof of Lemmas~\ref{lem-W},   \ref{lem-PhiPsi} and \ref{qw-lem}}\label{sec-ap}

 In this section, we present the proof of Lemmas~\ref{lem-W}, \ref{lem-PhiPsi} and \ref{qw-lem} for the sake of completeness.

\subsection{Proof of Lemma~\ref{lem-W}}\label{subap1}

It suffices to consider the case $\gamma>0$. Define the energy functional 
$\mathcal{E}_0: \mathrm{H}^1(\mathbb{R}_{+}) \to \mathbb{R}$ by
\[
\mathcal{E}_0[w]:=  \int_0^{\infty} \left(\frac{| w'|^2}{2}
+\frac{F(w)}{\mathcal{A}(q(0))} \right)\mathrm{d}t 
+\frac{1}{2\gamma} (w(0)-b_0)^2.
\]
Since $\mathcal{A}(q(0))>0$, assumptions \eqref{big-F} and \eqref{assume-f} 
imply that $\mathcal{E}_0$ is nonnegative and strictly convex on 
$\mathrm{H}^1(\mathbb{R}_{+})$. Consequently, by the direct method in the calculus 
of variations, $\mathcal{E}_0$ admits a unique minimizer $W \in \mathrm{H}^1(\mathbb{R}_{+}),$
which is a weak solution of \eqref{ode-W}. Moreover, by standard regularity theory for second-order equations, 
$W$ is smooth and hence a classical solution of \eqref{ode-W}. 
In particular, $W \in \mathrm{L}^\infty(\mathbb{R}_{+})$, 
and the solution is unique.

Note that \eqref{assume-f} together with $b_0>0$ and $\gamma\geq0$ holds. Therefore, by the strong maximum principle, we have $0< W \le b_0$ on $[0,\infty)$. Combining this with \eqref{assume-f} again, we deduce that $W$ is strictly convex and monotonically decreasing on $(0,\infty)$.

We now verify \eqref{w1w}. By multiplying \eqref{ode-W} by $W'$ and integrating, we obtain
\[
W'^2(t)=  \frac{2F(W(t))}{\mathcal{A}(q(0))}.
\]
Here we have used the facts that $|W(t)|+|W'(t)|\to0$ as $t\to\infty$ and that $F(0)=0$. Since $W'\leq0<W$ on $[0,\infty)$, it follows that
\begin{align}\label{w2w}
    W'(t)=-\sqrt{\frac{2F(W(t))}{\mathcal{A}(q(0))}}<0,\quad t\in[0,\infty).
\end{align}
Using \eqref{b0}, \eqref{w2w}, and the boundary condition of $W$ in \eqref{ode-W}, we obtain
\[
W(0)=b_* \quad\text{and}\quad W'(0)=-\sqrt{\frac{2F(b_*)}{\mathcal{A}(q(0))}}.
\]
Here we have used the fact that $b_*$ is the unique solution of \eqref{b0}, since $F$ is strictly increasing on $\mathbb{R}$. This proves \eqref{w1w}.

It remains to verify \eqref{w3w}. By \eqref{w2w} and $W(0)=b_*$, we have $W(t)\in(0,b_*]$ for all $t\in[0,\infty)$. Moreover, by \eqref{big-F} and \eqref{assume-f},
\[
F(W)\ge \left(\frac12\min_{[0,b_*]}f'\right)W^2>0.
\]
Together with \eqref{w2w}, this yields
\[
W'(t)+\left(\min_{[0,b_*]}\sqrt{\frac{f'}{\mathcal{A}(q(0))}}\right)W(t)\leq0,
\quad \forall\,t\geq0.
\]
Combining this inequality with $W(0)=b_*$, we obtain the estimate for $W$ in \eqref{w3w}. Furthermore, by \eqref{big-F} and \eqref{assume-f}, we have
\[
0\leq F(W(t))\leq (\max_{[0,b_*]}f)W(t)=f(b_*)W(t).
\]
Together with \eqref{w2w}, this gives the estimate for $W'$ in \eqref{w3w} and completes the proof of Lemma~\ref{lem-W}.
\subsection{Proof of Lemma~\ref{lem-PhiPsi}}\label{subap2}
The existence and uniqueness of classical solutions $\Phi$ and $\Psi$ follow from standard arguments, since equations \eqref{Phi-new} and \eqref{Psi-new} are linear and the variable coefficient $\frac{f'(W)}{\mathcal{A}(q(0))}$ is positive and uniformly bounded on $[0,\infty)$ by \eqref{assume-f} and Lemma~\ref{lem-W}. Therefore, we omit the details.

We now prove \eqref{iph-1}. By \eqref{Phi-new}, one can verify that
\[
-\frac12 \Phi'^2(t)
= \int_t^\infty \Phi''(s)\Phi'(s)\,\mathrm{d}s
= \int_t^\infty \left(\frac{f'(W(s))}{\mathcal{A}(q(0))}\Phi(s)\Phi'(s)
+ f(W(s))\Phi'(s)\right)\mathrm{d}s.
\]
On the other hand, by \eqref{ode-W} and \eqref{Phi-new}, we have
\begin{equation*}\begin{aligned}
-\frac12 \left(\Phi'(t)- \mathcal{A}(q(0)) W'(t)\right)^2
=& \int_t^\infty (\Phi- \mathcal{A}(q(0))W)''(s)
(\Phi- \mathcal{A}(q(0))W)'(s)\,\mathrm{d}s \\
=& \int_t^\infty \frac{f'(W(s))}{\mathcal{A}(q(0))} \Phi(s)
\left(\Phi'(s)- \mathcal{A}(q(0)) W'(s)\right)\,\mathrm{d}s \\
=& f(W(t))\Phi(t)
+ \int_t^\infty \left(\frac{f'(W(s))}{\mathcal{A}(q(0))}\Phi(s)\Phi'(s)
+ f(W(s))\Phi'(s)\right)\mathrm{d}s.
\end{aligned}\end{equation*}
Combining the above identities, we obtain
\[
\frac12 \Phi'^2(t)
-\frac12 \left(\Phi'(t)- \mathcal{A}(q(0)) W'(t)\right)^2
= f(W(t))\Phi(t),
\]
that is,
\begin{align}\label{1208-1}
\Phi'(t)-\frac{\mathcal{A}(q(0))}{2}W'(t)
= \frac{f(W(t))}{\mathcal{A}(q(0))W'(t)}\Phi(t).
\end{align}
Here we have used the fact that $W'(t)\neq0$ for all $t\in[0,\infty)$. Hence, \eqref{iph-1} follows immediately from \eqref{w1w} and \eqref{1208-1}.

Recall that $W(0)=b_*$. Setting $t=0$ in \eqref{iph-1} and using the boundary condition of $\Phi$ at $t=0$, we can derive \eqref{iph-2}. In particular, \eqref{iph-1} and \eqref{iph-2} imply that $\Phi<0$ and is uniformly bounded on $[0,\infty)$. Indeed, \eqref{iph-2} yields $\Phi(0)<0$, and $\Phi$ is decreasing in a neighborhood of $t=0$. Suppose, by contradiction, that $\Phi$ attains a nonnegative maximum at some interior point $t_M\in(0,\infty)$. Then, by \eqref{iph-1},
\[
0\leq \frac{f(W(t_M))}{\sqrt{2\mathcal{A}(q(0))F(W(t_M))}}\Phi(t_M)
= -\sqrt{\frac{\mathcal{A}(q(0))F(W(t_M))}{2}}<0,
\]
since $0<W(t_M)\leq b_*$ implies $F(W(t_M))>0$. This contradiction shows that $\Phi<0$ on $[0,\infty)$. Consequently, $\Phi$ attains its negative minimum at an interior point. Combining this fact with \eqref{w1w}--\eqref{w3w} and \eqref{iph-1}, we conclude that $\Phi'$ is uniformly bounded on $[0,\infty)$. This completes the proof of (i).

We next prove \eqref{ips-1}. By \eqref{ode-W}, \eqref{Psi-new}, and \eqref{w1w}, we observe that
\begin{align*}
0=& \int_t^\infty \left(W''(s)-\frac{f(W(s))}{\mathcal{A}(q(0))}\right)\Psi'(s)\,\mathrm{d}s\\
=& -W'(t)\Psi'(t)
-\int_t^\infty \left(W'(s)\Psi''(s)
+ \frac{f(W(s))}{\mathcal{A}(q(0))}\Psi'(s)\right)\mathrm{d}s \\
=& -W'(t)\Psi'(t)
-\int_t^\infty \left[\left(\frac{f(W(s))}{\mathcal{A}(q(0))}\Psi(s)\right)'
-\sqrt{\frac{2F(W(s))}{\mathcal{A}(q(0))}}W'(s)\right]\mathrm{d}s \\
=& \sqrt{\frac{2F(W(t))}{\mathcal{A}(q(0))}}\Psi'(t)
+ \frac{f(W(t))}{\mathcal{A}(q(0))}\Psi(t)
- \int_0^{W(t)}\sqrt{\frac{2F(s)}{\mathcal{A}(q(0))}}\,\mathrm{d}s,
\end{align*}
where we have used the facts that $|\Psi(t)|+|\Psi'(t)|\to0$ as $t\to\infty$ and
\[
\int_t^\infty \sqrt{\frac{2F(W(s))}{\mathcal{A}(q(0))}} W'(s)\,\mathrm{d}s
= \int_{W(t)}^0 \sqrt{\frac{2F(s)}{\mathcal{A}(q(0))}}\,\mathrm{d}s.
\]
Consequently,
\[
\Psi'(t)
+ \frac{f(W(t))}{\sqrt{2\mathcal{A}(q(0))F(W(t))}}\Psi(t)
= \sqrt{\frac{\mathcal{A}(q(0))}{2F(W(t))}}
\int_0^{W(t)}\sqrt{\frac{2F(s)}{\mathcal{A}(q(0))}}\,\mathrm{d}s,
\]
which gives \eqref{ips-1}.

Setting $t=0$ in \eqref{ips-1} and using the boundary condition of $\Psi$ at $t=0$, we obtain \eqref{ips-2}. Moreover, since $f(W(t))>0$ and $G(W(t))>0$ on $[0,\infty)$, a similar argument to that in (i) shows that $\Psi>0$ on $[0,\infty)$ and that $\sup_{[0,\infty)}(|\Psi|+|\Psi'|)<\infty$.

It remains to prove \eqref{cw-exp}. Note that $-\Phi>0$ on $[0,\infty)$. Applying \eqref{w1w} to \eqref{iph-1}, we obtain
\begin{align}\label{12-esphi}
(-\Phi)'(t)+M_1(-\Phi(t))
\leq \sqrt{\frac{b_*f(b_*)\mathcal{A}(q(0))}{2}}
\mathrm{e}^{-\frac{M_0}{2}t},
\quad t\geq0,
\end{align}
where
\[
M_1:=\inf_{(0,b_*]}\frac{f}{\sqrt{2\mathcal{A}(q(0))F}}>0,
\]
since $\lim_{t\to0^+}\frac{f(t)}{\sqrt{2\mathcal{A}(q(0))F(t)}}
=\sqrt{\frac{f'(0)}{\mathcal{A}(q(0))}}>0$.
From \eqref{12-esphi}, we have
\[
(-\Phi(t)\mathrm{e}^{M_1t})'
\leq \sqrt{\frac{b_*f(b_*)\mathcal{A}(q(0))}{2}}
\mathrm{e}^{(M_1-\frac{M_0}{2})t}.
\]
Assuming $M_1-\frac{M_0}{2}\neq0$, this yields
\begin{align}\label{adepphi}
0<-\Phi(t)
\leq \widehat{C}\mathrm{e}^{-\min\{M_1,\frac{M_0}{2}\}t},
\quad t\geq0,
\end{align}
with
\[
\widehat{C}
=|\Phi(0)|
+\frac{\sqrt{2b_*f(b_*)\mathcal{A}(q(0))}}{|2M_1-M_0|}.
\]
Together with \eqref{iph-1}, this implies
$|\Phi'(t)|\leq \widehat{C'}\mathrm{e}^{-\min\{M_1,\frac{M_0}{2}\}t}$
for some constant $\widehat{C'}>\widehat{C}$. Applying the same argument to \eqref{ips-1}--\eqref{ips-2}, we obtain
\[
\Psi(t)+|\Psi'(t)|
\leq \widehat{C''}\mathrm{e}^{-\min\{M_1,\frac{M_0}{2}\}t}
\]
for some positive constant $\widehat{C''}$. Combining these estimates yields \eqref{cw-exp} and completes the proof of Lemma~\ref{lem-PhiPsi}.
\begin{remark}\label{rk-ph}
One can solve linear equations \eqref{iph-1}--\eqref{iph-2} and \eqref{ips-1}--\eqref{ips-2}  with simple calculations as follows:
\begin{align}
   \Phi(t)=&\, -\frac{\gamma \mathcal{A}(q(0)) F(b_*)\mathrm{e}^{-\mathfrak{X}(t)}}{\gamma f(b_*)+\sqrt{2\mathcal{A}(q(0)) F(b_*)}}-\int_0^t\sqrt{\frac{\mathcal{A}(q(0))F(W(r))}{2}}\mathrm{e}^{\mathfrak{X}(r)-\mathfrak{X}(t)}~\mathrm{d}r,\label{ish}\\
   \Psi(t)= &\,\frac{\gamma\mathcal{A}(q(0))G(b_*)\mathrm{e}^{-\mathfrak{X}(t)}}{\gamma f(b_*)+ \sqrt{2\mathcal{A}(q(0))F(b_*)}}+\int_0^t\sqrt{\frac{\mathcal{A}(q(0))}{2F(W(r))}}G(W(r))\mathrm{e}^{\mathfrak{X}(r)-\mathfrak{X}(t)}~\mathrm{d}r,
   \label{ihs}
\end{align}
where $\mathfrak{X}(r)=\int_0^r\frac{f(W(s))\,\mathrm{d}s}{\sqrt{2\mathcal{A}(q(0))F(W(s))}}$.
\end{remark}
\subsection{Proof of Lemma~\ref{qw-lem}}\label{subap3}
By \eqref{Q-F} and \eqref{w1w}, we compute
\begin{equation*}
\begin{aligned}
\int_0^{\frac{d_{\eps}}{\eps}} \left(q(W(t))- q(0) \right)\,\mathrm{d}t
=& -\int_0^{\frac{d_{\eps}}{\eps}} \left(q(W(t))-q(0) \right)
\sqrt{\frac{\mathcal{A}(q(0))}{2F(W(t))}}W'(t)\,\mathrm{d}t \\
=& \int_{W(\frac{d_{\eps}}{\eps})}^{b_*}
\left(q(s)-q(0)\right)
\sqrt{\frac{\mathcal{A}(q(0))}{2F(s)}}\,\mathrm{d}s \\
=& \sqrt{\mathcal{A}(q(0))}
\left(\mathcal{Q}_F(b_*)
-\int_{0}^{W(\frac{d_{\eps}}{\eps})}
\frac{q(s)-q(0)}{\sqrt{2F(s)}}\,\mathrm{d}s \right) \\
=& \sqrt{\mathcal{A}(q(0))}\,\mathcal{Q}_F(b_*)
+O\!\left(\mathrm{e}^{-M_0\frac{d_{\eps}}{\eps}}\right),
\end{aligned}
\end{equation*}
which gives \eqref{0qW}. Here we have used the fact that $W(t)\in(0,b_*]$ and that
\[
\sup_{s\in(0,b_*]}
\left|\frac{q(s)-q(0)}{\sqrt{2F(s)}}\right|<\infty,
\]
since by \eqref{big-F} and \eqref{assume-f} one has
\[
\frac{q(s)-q(0)}{\sqrt{2F(s)}}\xrightarrow{s\to0^+}
\frac{q'(0)}{\sqrt{f'(0)}}.
\]
This also implies the uniform boundedness of $\mathcal{Q}_F'(W(t))$ on $[0,\infty)$. Consequently,
\[
\left|\int_{0}^{W(\frac{d_{\eps}}{\eps})}
\frac{q(s)-q(0)}{\sqrt{2F(s)}}\,\mathrm{d}s\right|
\lesssim |W(\tfrac{d_{\eps}}{\eps})|.
\]

Next, we compute
\begin{equation}\label{1twq}
\begin{aligned}
\int_0^{\frac{d_{\eps}}{\eps}} t\left(q(W(t))-q(0)\right)\,\mathrm{d}t
=& -\sqrt{\mathcal{A}(q(0))}
\int_0^{\frac{d_{\eps}}{\eps}}
\frac{t\left(q(W(t))-q(0)\right)}{\sqrt{2F(W(t))}}W'(t)\,\mathrm{d}t \\
=& -\sqrt{\mathcal{A}(q(0))}
\int_0^{\frac{d_{\eps}}{\eps}}
t\left(\mathcal{Q}_F(W(t))\right)' \mathrm{d}t \\
=& \sqrt{\mathcal{A}(q(0))}
\int_0^\infty \mathcal{Q}_F(W(t))\,\mathrm{d}t
+O\!\left(\frac{d_{\eps}}{\eps}\mathrm{e}^{-M_0\frac{d_{\eps}}{\eps}}\right),
\end{aligned}
\end{equation}
and
\begin{equation}\label{2twq}
\begin{aligned}
&\int_0^{\frac{d_{\eps}}{\eps}} t^2\left(q(W(t))- q(0) \right)\,\mathrm{d}t\\
=& -\sqrt{\mathcal{A}(q(0))}
\int_0^{\frac{d_{\eps}}{\eps}} t^2
\left(\mathcal{Q}_F(W(t))\right)' \mathrm{d}t \\
=& -\sqrt{\mathcal{A}(q(0))}
\Bigg(
\frac{d_{\eps}^2}{\eps^2}\mathcal{Q}_F\!\left(W(\tfrac{d_{\eps}}{\eps})\right)
+2\int_0^{\frac{d_{\eps}}{\eps}}
t\mathcal{Q}_F(W(t))
\sqrt{\frac{\mathcal{A}(q(0))}{2F(W(t))}}W'(t)\,\mathrm{d}t
\Bigg) \\
=&~ O\left(\frac{d_{\eps}^2}{\eps^2}
\mathrm{e}^{-M_0\frac{d_{\eps}}{\eps}}\right)
-2\mathcal{A}(q(0))
\int_0^{\frac{d_{\eps}}{\eps}}
t\left(\widetilde{\mathcal{Q}}_F(W(t))\right)' \mathrm{d}t \\
=& ~O\left(\frac{d_{\eps}^2}{\eps^2}
\mathrm{e}^{-M_0\frac{d_{\eps}}{\eps}}\right)
+2\mathcal{A}(q(0))
\int_0^{\infty}\widetilde{\mathcal{Q}}_F(W(t))\,\mathrm{d}t.
\end{aligned}
\end{equation}
These yield \eqref{1qW} and \eqref{2qW}. Here we have used \eqref{Q-F} and \eqref{w1w} to verify that
\[
|\mathcal{Q}_F(W(t))|
\leq b_*\!\left(\sup_{(0,b_*]} |\mathcal{Q}_F'|\right)
\mathrm{e}^{-M_0t},
\quad
|\widetilde{\mathcal{Q}}_F(W(t))|
\leq b_*\!\left(\sup_{(0,b_*]}
\frac{|\mathcal{Q}_F|}{\sqrt{2F}}\right)
\mathrm{e}^{-M_0t}.
\]
The remainder estimates in \eqref{1twq} and \eqref{2twq} then follow from the identities
\[
\int_0^{\frac{d_{\eps}}{\eps}} t\left(\mathcal{Q}_F(W(t))\right)' \mathrm{d}t
= \frac{d_{\eps}}{\eps}\mathcal{Q}_F\!\left(W(\tfrac{d_{\eps}}{\eps})\right)
-\int_0^{\frac{d_{\eps}}{\eps}} \mathcal{Q}_F(W(t))\,\mathrm{d}t,
\]
and
\[
\int_0^{\frac{d_{\eps}}{\eps}} t\left(\widetilde{\mathcal{Q}}_F(W(t))\right)' \mathrm{d}t
= O\!\left(\frac{d_{\eps}}{\eps}\mathrm{e}^{-M_0\frac{d_{\eps}}{\eps}}\right)
-\int_0^{\infty}\widetilde{\mathcal{Q}}_F(W(t))\,\mathrm{d}t.
\]
Note that $\frac{|\mathcal{Q}_F(s)|}{\sqrt{2F(s)}}\to\frac{|q'(0)|}{f'(0)}$ as $s\to0^+$.

It remains to prove part (ii). By \eqref{w3w} and \eqref{cw-exp}, we have
\begin{align*}
\left|\int_{\frac{d_{\eps}}{\eps}}^{\infty}
t^k q'(W(t))\Phi(t)\,\mathrm{d}t\right|
\leq \widetilde{C}
\left(\sup_{(0,b_*]}|q'|\right)
\int_{\frac{d_{\eps}}{\eps}}^{\infty}
t^k\mathrm{e}^{-\widetilde{M}t}\,\mathrm{d}t
= O\!\left(\frac{d_{\eps}^k}{\eps^k}
\mathrm{e}^{-\widetilde{M}\frac{d_{\eps}}{\eps}}\right),
\end{align*}
which proves \eqref{kqh}. The estimate \eqref{kqs} follows from the same argument, and the proof of Lemma~\ref{qw-lem} is complete.

\subsection*{\bf Acknowledgments.} 
The research of C.-C. Lee was partially supported by the grant 114-2115-M-007-013-MY2 of the Ministry of Science and Technology of Taiwan.  
S.H. Moon was supported by the National Research Foundation of Korea grant funded by the Ministry of Science and ICT (No. RS-2022-NR072398). W. Yang is supported by National Key R\&D Program of China 2022YFA1006800, NSFC No. 12271369 and 12531010, FDCT No. 0070/2024/RIA1, Start-up Research Grant No. SRG2023-00067-FST, Multi-Year Research Grant No. MYRG-GRG2024-00082-FST-UMDF, MYRG-GRG2025-00051-FST and UMDF No. TISF/2025/006/FST.

\end{document}